\documentclass{article}
\usepackage{latexsym}
\usepackage{leftidx}
\usepackage{slashed}
\usepackage{amsmath,amsthm, pdfpages}
\usepackage{enumerate}
\usepackage{amssymb}
\usepackage{upgreek}

\usepackage[margin=1.3in,dvips]{geometry}

\newcommand\numberthis{\addtocounter{equation}{1}\tag{\theequation}}

\newtheorem{Lem}{Lemma}[section]

\newtheorem{Prop}{Proposition}[section]
\newtheorem{thm}{Theorem}

\newtheorem{Remark}{Remark}[section]

\def\R{\mathbb R}

\def\f12{\frac 1 2}

\def\a{\alpha}
\def\b{\beta}

\def\vep{\varepsilon}
\def\la{\lambda}
\def\La{\Lambda}

\def\om{\omega}

\def\pa{\partial}

\def\R{\mathbb{R}}

  \allowdisplaybreaks

\begin{document}

\title{On the 3D Relativistic
Vlasov-Maxwell System with large Maxwell field}

\author{Dongyi Wei \and Shiwu Yang}

\AtEndDocument{ {\footnotesize%
  \addvspace{\medskipamount}
  \textsc{$\quad \quad$School of Mathematical Sciences, Peking University, Beijing, China} \par
  \textit{E-mail address}: \texttt{jnwdyi@pku.edu.cn} \par

  \addvspace{\medskipamount}
  \textsc{Beijing International Center for Mathematical Research, Peking University, Beijing, China} \par
  \textit{E-mail address}: \texttt{shiwuyang@math.pku.edu.cn}
  }}

\date{}

  \maketitle

\begin{abstract}
This paper is devoted to the study of relativistic Vlasov-Maxwell system in three space dimension. For a class of large initial data, we prove the global existence of classical solution with sharp decay estimate. The initial Maxwell field is allowed to be arbitrarily large and the initial density distribution is assumed to be small and decay with rate $(1+|x|+|v|)^{-9-}$. In particular, there is no restriction on the support of the initial data.
\end{abstract}

 \section{Introduction}
  In this paper, we study the Cauchy problem to the relativistic Vlasov-Maxwell (RVM) system
  \begin{align}
  \label{RVM}
  \left\{\begin{array}{l}
   \partial_t f+\hat{v}\cdot\nabla_xf+(E+\hat{v}\times B)\cdot\nabla_vf=0,\\
   \partial_t E-\text{curl}B=-4\pi j, \quad \text{div}E=4\pi \rho,\\
    \partial_tB+\text{curl}E=0, \phantom{-\pi j}\quad \text{div}B=0,\\
E(0, x)=E_0,\quad B(0, x)=B_0,\quad f(0, x, v)=f_0(x, v)
    \end{array}\right.
\end{align}
in $\mathbb{R}^{1+3}$, where $f(t,x,v)$ denotes the density distribution for the particles and $\hat{v}$ is the relativistic speed
\[
\hat{v}=\frac{v}{\sqrt{1+|v|^2}},\quad v\in \mathbb{R}^3.
\]
The total charge density $\rho$ and total current density $j$ are given by
\begin{align}
\label{rhoj}
 &\rho(t,x)=\int_{\R^3}f(t,x,v)dv,\quad j(t,x)=\int_{\R^3}\hat{v}f(t,x,v)dv,
\end{align}
which determine the electromagnetic field $(E,B)$.
The admissible initial data set $(E_0, B_0, f_0)$ must verify the compatibility condition
\begin{align*}
\text{div}E_0=4\pi \rho_0=4\pi\int_{\R^3}f_0(x,v)dv,\quad \text{div}B_0=0.
\end{align*}

The RVM system is used to describe the evolution of collisionless particles driven by electromagnetic field in plasma physics and has drawn extensive attention in the past decades. Existence and uniqueness of local classical solutions are known since the work \cite{wollman84:local:VM} of Wollman. Global weak solution has been constructed by DiPerna-Lions in \cite{Lions89:globalweak:VM}. The important problem is whether or not the classical smooth solution can be propagated for all time, which still remains open.
In the seminal work \cite{Glassey:VM:singularity:86} of
Glassey-Strauss, they showed that the uniform boundedness of the support of the velocity for bounded times implies global $C^1$ solutions, which can be viewed as a continuation criteria for RVM. Recently, it is of increasing interest to look for new methods and new breakdown criteria in order to tackle this problem, see
  \cite{Glassey87:VM:Hivelocity}, \cite{Glassey89:3DVM:largevelocity}, \cite{Illner10:VM}, \cite{Luk14:VM:criterion}, \cite{Luk:VlasovMaxwell:Strichartz}, \cite{Pallar15:VM:criteria}, \cite{Patel18:newcriteria:VM} and references therein.

Nevertheless, global regularity holds with symmetry or certain smallness on the field. Symmetry (translation symmetry or spherical symmetry) reduces the system to the associated one in lower dimensions, which has been well understood as in \cite{Glassey90:1andhalfD:VM}, \cite{Horst90:VM:sphericalSy}, \cite{Glassey:VM:2.5D}, \cite{Glassey98:VM:2D:1},
\cite{Luk:VlasovMaxwell:Strichartz}, \cite{Wang20:3D:VM:large}. Smallness allows one to use perturbative method to study the evolution of the solution. First such result was contributed by Glassey-Strauss in \cite{Glassey87:VM:absence}, where it has been shown that compactly supported small initial data lead to a unique global classical solution of RVM. This result had shortly been improved to almost neutral data, allowing the density distribution to be large in \cite{Glassey88:VM:neutral}. Another interesting generalization was established by Rein in \cite{Rein90:VM}, where he proved that given large solution verifying certain decay property is stable under small perturbation. A direct consequence of this result is the global existence of nearly spherical symmetric solutions. However all these results heavily relied on the support of the velocity and hence were only applicable to compactly supported initial data. This even excludes the physical case of particles with Gaussian distribution.

The first successful attempt to allow arbitrarily large velocity was obtained by
Schaeffer in \cite{Schaeffer:VM:notcompat:p}. He demonstrated that the solution stays regular for all time as long as the initial density distribution decays sufficiently fast in velocity (with decay rate $(1+|v|)^{-60-12\sqrt{17}-}$). However it was still assumed that the initial data are small and compactly supported in the spatial variable.
To completely remove the restriction on the support of the initial data in the scope of small data regime, Fourier analysis (\cite{Klainerman02:VM:Fourier}, \cite{Wang18:3D:VM}) and vector field method (\cite{Jacques17:Vect:Vlasov}, \cite{Leo18:3D:VM}) have recently been used to study the global dynamics of RVM.
These new methods are robust in the sense that they also give the sharp decay estimates for the derivatives of the solutions, which in the meanwhile means that it requires higher regularity of the initial data and the arguments are more involved.

The goal of this paper consists mainly two folds: Firstly we are trying to construct global large solutions to RVM with asymptotic decay properties. This is partially inspired by the work \cite{Rein90:VM} of Rein and the global stability of Maxwell field coupled to scalar fields \cite{yangMKG}, \cite{Yang:mMKG:3D:largeMaxwell}. We show that linear Maxwell field is also stable under perturbation of Vlasov field, that is, smallness is only required on the density distribution and the Maxwell field can be arbitrarily large. Secondly, we remove the restriction on the support of the initial data and improve the decay rate of the density distribution under the frame work established early by Glassey-Strauss (using the representation formula for linear solutions).

More precisely, we prove in this paper the following result.
\begin{thm}
\label{thm:main}
Consider the Cauchy problem to the relativistic Vlasov-Maxwell equation \eqref{RVM} with admissible initial data $(E_0, B_0, f_0)$ such that $E_0, B_0\in C^2,\ f_0\in C^1,\ f_0\geq 0$ and
\begin{align*}
|\nabla^kE_0(x)|+|\nabla^kB_0(x)|&\leq M(1+|x|)^{-2-k},\quad\forall  k=0,1,2,\\
|f_0(x,v)|&\leq \varepsilon_0(1+|x|+|v|)^{-q}
\end{align*}
for some positive constants $M>1,\ q>9$, $\vep_0$.  Then for sufficiently small $\varepsilon_0$ relying only on $M$ and $q$, the equation \eqref{RVM} admits a global $C^1$ solution $(E(t, x), B(t, x), f(t, x, v))$ verifying the following decay estimates
 \begin{align*}
 |E(t,x)|+|B(t,x)| &\leq CM (1+\left|t-|x|\right|)^{-1}(1+t+|x|)^{-1},\\
 \int_{\mathbb{R}^3}f(t, x, v)dv &\leq C(M) \varepsilon_0 (1+t+|x|)^{-3}
\end{align*}
for all $t\geq 0,\ x\in\R^3$ with some constant $C$ depending only on $q$ and constant $C(M)$ depending also on $M$.
\end{thm}

We give several remarks.

\begin{Remark}
We emphasize that the smallness is only required on the $C^0$ norm of the initial density distribution instead of  the $C^1$ norm in the previous works.
\end{Remark}
\begin{Remark}
Compared to the rapid decay in velocity of the initial density distribution obtained by Schaeffer, we improve the decay rate to $(1+|x|+|v|)^{-9-}$.
\end{Remark}

\begin{Remark}
Same argument also applies to the stability of given large solutions of RVM with appropriated decay estimates. We hence can extend the result of Rein in \cite{Rein90:VM} to general unrestricted data.
\end{Remark}



The proof relies on the frame work introduced by Glassey-Strauss in \cite{Glassey:VM:singularity:86}, \cite{Glassey87:VM:Hivelocity}, \cite{Glassey89:3DVM:largevelocity}. The solution to the nonlinear RVM can be constructed via the standard iteration scheme by starting with
$$E^{(0)}(t,x)=B^{(0)}(t,x)=0.$$
Given $(E^{(n-1)}$, $B^{(n-1)})$, define $ f^{(n)}$ as solution
of the linear Vlasov equation
\begin{align*}
 \partial_tf^{(n)}+\hat{v}\cdot\nabla_xf^{(n)}+(E^{(n-1)}&+\hat{v}\times B^{(n-1)})\cdot\nabla_vf^{(n)}=0,\quad
 f^{(n)}(0,x,v)=f_0(x,v),
\end{align*}
which can be solved by using the method of characteristics
\begin{align*}
 f^{(n)}(t,x,v)=f_0(X(0,t,x,v),V(0,t,x,v)).
\end{align*}
Here $X=X(s,t,x,v)$, $V=V(s,t,x,v)$ are solutions of the ODE
\begin{equation}
\label{eq:cha}
\begin{cases}
 &\partial_sX=\hat{V}= V/\sqrt{1+|V|^2},\quad \partial_sV=E^{(n-1)}(s,X)+\hat{V}\times B^{(n-1)}(s,X),\\
 & X(t,t,x,v)=x,\ V=V(t,t,x,v)=v,\quad 0\leq s\leq t.
 \end{cases}
\end{equation}
Then define the associated charge density and current density
\begin{align*}
\rho^{(n)}(t,x)=\int_{\R^3}f^{(n)}(t,x,v)dv,\quad
j^{(n)}(t,x)=\int_{\R^3}\hat{v}f^{(n)}(t,x,v)dv.
\end{align*}
These density functions then lead to the new Maxwell fields
 $(E^{(n)}$, $B^{(n)} )$ as solutions of
 \begin{align*}
 &\partial_tE^{(n)}-\text{curl}B^{(n)}=-4\pi j^{(n)},\quad \text{div}E^{(n)}=4\pi \rho^{(n)},\\
 &\partial_tB^{(n)}+\text{curl}E^{(n)}=0,\phantom{-\pi j^{(n)}}\quad \text{div}B^{(n)}=0,\\
 & E^{(n)}(0,x)=E_0(x),\quad B^{(n)}(0,x)=B_0(x).
\end{align*}
The key observation of Glassey-Strauss is that the Maxwell field can be expressed as integrals of the density distribution and the convergence of this iteration sequence is guaranteed by the uniform boundedness of the kinetic energy density.
\begin{thm}[Glassey-Strauss \cite{Glassey89:3DVM:largevelocity}]
\label{thm:strauss89}
For each $T>0$, if the kinetic energy density verifies the condition
 $$\int_{\R^3}|{v}|f^{(n)}(t,x,v)dv\leq C,\quad \forall 0\leq t\leq T,\quad \forall n$$
 for some constant $C>0$, then the above iteration sequence $(E^{(n)}, B^{(n)}, f^{(n)})$ converges to a solution $(E, B, f)$ of the nonlinear system RVM as $n\rightarrow \infty$.
\end{thm}
 From this theorem, to construct a solution of RVM, it is important to control the density distribution, which relies on the characteristics. This requires that the Maxwell field lies in certain function space.
 Abbreviating the Maxwell field as $K(t,x)=(E(t,x),B(t,x))$, define the norm
\begin{align*}
&\|K\|_0=\sup_{x,t}(\left|t-|x|\right|+1)(t+|x|+1)(|E(t,x)|+|B(t,x)|).
\end{align*}
The decay factor is consistent with the decay property of the linear Maxwell field, which plays the role that the characteristics will not diverge too far from the linear ones. To ensure that the sequence of Maxwell field $(E^{(n)}, B^{(n)})$ lies in the same function space, we also need higher order regularity. For this purpose, define the norm
\begin{align*}
&\|K\|_{\gamma}=\sup_{|x|\leq |y|\leq t}(t-|y|+1)^{\gamma}(t+1)
\frac{|E(t,x)-E(t,y)|+|B(t,x)-E(t,y)|}{|x-y|+1},\quad \gamma>1.
\end{align*}
Compared to the previous works, we emphasize here that this norm is only used in the interior region $\{|x|\leq t\}$ and the $\|K\|_0$ norm is sufficient to close the argument in the exterior region $\{|x|\geq t\}$. Moreover, we note that the $\|K\|_{\gamma}$ norm is weaker than the weighted Lipschitz norm or the weighted $C^1$ norm used in the previous works. This allows us to improve the smallness assumption on the initial density distribution from $C^1$ norm to $C^0$ norm.

Choose constants $\a, \b$ such that
$$0<\alpha<\min\{\frac{1}{6}\beta,\frac{1}{2}(q-9)\beta\},\quad 0<\b<1,$$
where $q>9$ is the constant recording the decay rate of the initial density distribution.
Define
\begin{equation}
\label{eq:def4:Knorm}
\|K\|=\|K\|_0+\|K\|_{1+\alpha}.
\end{equation}
We will work with Maxwell fields in the following function space
\begin{align*}
&\mathcal{K}_{\Lambda}\equiv\{K|K\in C^1(\mathbb{R}^{1+3}),\quad \|K\|_0\leq \Lambda,\quad \|K\|_{1+\alpha}\leq \Lambda^{2}\}
\end{align*}
for some large constant $\Lambda$ depending linearly on the size $M$ of the initial Maxwell field. With such Maxwell field, we first obtain improved bounds on the associated characteristics (equation \eqref{eq:cha}) both on the spatial position $X$ and the relativistic speed $\hat{V}$. This gives a rough pointwise decay estimate for the density distribution. For rapid decaying data as shown by Schaeffer, this estimate is sufficient to bound the kinetic energy density as well as those source terms for the Maxwell equations (see the representation formula \eqref{eq:def:electric}). To improve the assumption on the initial data, it requires refined estimates on the velocity in terms of the initial density distribution. The key new observation of this work is to dyadicly decompose the initial density distribution
\begin{align*}
f_0= \sum\limits_{k=1}^{\infty} f_{0, k},\quad f_{0, k}= f_0 \psi_{k-1}
\end{align*}
with $\psi_{k}$ smooth and supported on $\{|x|+|v|\leq 2^{k+1}\}$. This enables us to derive refined estimate on each component $f_{[k]}^{(n)}$ (solving the Vlasov equation with data $f_{0, k}$). The decay assumption on the initial density distribution (the rate $q>9$) is required so that all the pieces $f_{[k]}^{(n)}$ as well as the associated electromagnetic fields under this decomposition are summable in terms of $k$.

The bound on the Maxwell field becomes standard once we have understood the density distribution. However, since the Maxwell field is large, we split the full Maxwell field into the linear part and the perturbation part. Define the linear Maxwell field
\begin{align}
\label{defE}
&\mathcal{E}(t,x)=\frac{1}{4\pi t^2}\int_{|y-x|= t}[E_0(y)+((y-x)\cdot\nabla)E_0(y)+t\text{curl}B_0(y))dS_y,\\
 \notag
 &\mathcal{B}(t,x)=\frac{1}{4\pi t^2}\int_{|y-x|= t}[B_0(y)+((y-x)\cdot\nabla)B_0(y)-t\text{curl}E_0(y))dS_y.
\end{align}
We first show that the $K$-norm $\|(\mathcal{E}, \mathcal{B})\|$ is bounded by a constant multiple of the initial size $M$. Then to show that the Maxwell field $(E^{(n)}, B^{(n)})$ lies in the space $\mathcal{K}_{\Lambda}$, it suffices to demonstrate that the full Maxwell field is close to the linear part $(\mathcal{E}, \mathcal{B})$ when the initial density distribution is sufficiently small.

The plan of the paper is as follows: In section 2, we bound the characteristics associated to Maxwell field in $\mathcal{K}_{\Lambda}$. Then in section 3, we decompose the initial density distribution and derive necessary estimates on the averaged density. With these preparations, we carry out estimates on the Maxwell fields in Section 4. The last section is devoted to proof of the main theorem.

\textbf{Acknowledgments.} The author would like to thank G. Rein for enlightening and helpful discussions. S. Yang is partially supported by NSFC-11701017.

\section{Improved bounds on the characteristics}
To insure the convergence of the iteration sequence, it is important to bound the kinetic energy density, which relies on the behavior of the characteristics. The aim of this section is to control the characteristics associated to given Maxwell field with finite $K$-norm. We begin with a Lemma controlling the evolution of the velocity.

For notational ease, in the following, we take $K=(E,B)=(E^{(n-1)},B^{(n-1)})$, $f=f^{(n)}$ and $K^*=(E^*,B^*)=(E^{(n)},B^{(n)}).$
\def\va{\varphi_0}\def\vb{\varphi_1}
\begin{Lem}
\label{lem:imbd:cha}
Let $K=(E, B)$ such that $\|K\|<+\infty$. Let $(X, V)$ be the associated characteristic solving the ODE \eqref{eq:cha} on the fixed time interval $[0, T^*]$.
Then for $0\leq s\leq t\leq T^*$, we have
\begin{align*}
 |V(t)|+ \|K\|_0\ln(2+ \|K\|_0)&\leq C(|V(s)|+ \|K\|_0\ln(2+ \|K\|_0)),\\|V(s)|+ \|K\|_0\ln(2+ \|K\|_0)&\leq C(|V(t)|+ \|K\|_0\ln(2+ \|K\|_0))
\end{align*}
for some constant $C$ independent of $T^*$, $t$, $s$.
\end{Lem}
\begin{proof}
Fixed $(T^*, x, v)\in \mathbb{R}^{1+3+3}$. Define
\begin{align*}
&U(t)=\sup\limits_{0\leq s\leq t}\{|V(s)|\},\quad \hat{U}(t)=U(t)/\sqrt{1+U^2(t)}.
\end{align*}
In view of the characteristic equations \eqref{eq:cha}, we derive that
\begin{align*}
&|X(s)-X(0)|\leq s\hat{U}(t)\leq s,
\end{align*}
from which we conclude that
\begin{align*}
  &|X(s)|\leq |X(0)|+s\hat{U}(t),\quad |X(s)|+s\geq  \max(|X(0)|,s),\\
  &\left|s-|X(s)|\right|\geq \max (s-|X(s)|, 0)\geq   \max(s-|X(0)|-s\hat{U}(t),0)
\end{align*}
for $ 0\leq s\leq t.$  Let $$R_1=|X(0)|,\quad R_2=|X(0)|/(1-\hat{U}(t)).$$
Hence for $t_1,t_2\in[0,t],$ we can show that
\begin{align*}
 |V(t_1)-V(t_2)|&\leq \int_0^t\frac{\|K\|_0ds}{(\left|s-|X(s)|\right|+1)(s+|X(s)|+1)}\\
 &\leq\int_0^t\frac{\|K\|_0ds}{(\max(s-|X(s)|,0)+1)(\max(|X(0)|,s)+1)}
\\
&\leq\int_0^{R_1}\frac{\|K\|_0ds}{|X(0)|+1}+\int_{R_1}^{R_2}\frac{\|K\|_0ds}{s+1}
+\int_{R_2}^{+\infty}\frac{\|K\|_0}{(s-|X(0)|-s\hat{U}(t)+1)(s+1)}
\\
&=\frac{|X(0)|\|K\|_0}{|X(0)|+1}+{\|K\|_0}\ln\frac{R_2+1}{R_1+1}+\|K\|_0{\va(|X(0)|-\hat{U}(t))},
\end{align*}
where
\begin{align*}
\va(z)=(1/z)\ln(1+z),\quad \va(0)=1,\quad -1<z<\infty.
\end{align*}
In particular $ \va$ is decreasing and
 \begin{align*}
 \va(-z)=\sum_{n=0}^{+\infty}\frac{z^n}{n+1}\leq 1+\sum_{n=1}^{+\infty}\frac{z^n}{n}=1-\ln(1-z),\quad 0\leq z<1.
 \end{align*}
 This implies that
 \begin{align*}
 \va(|X(0)|-\hat{U}(t))\leq \va(-\hat{U}(t))\leq 1-\ln(1-\hat{U}(t)).
 \end{align*}
 On the other hand, recall that
 \begin{align*}
 R_2=R_1/(1-\hat{U}(t))\geq 0, \quad  (R_2+1)/(R_1+1)\leq 1/(1-\hat{U}(t)).
 \end{align*}
 We therefore can further bound that
 \begin{align*}
 |V(t_1)-V(t_2)|&\leq \|K\|_0+{\|K\|_0}\ln[1/(1-\hat{U}(t))]+\|K\|_0[1-\ln(1-\hat{U}(t))]\\&=2\|K\|_0[1-\ln(1-\hat{U}(t))]\\&=2\|K\|_0[1+\ln(1+\hat{U}(t))+\ln(1+{U}^2(t))]
\\&\leq 2\|K\|_0[1+\ln 2+\ln(1+(8\|K\|_0)^2)]+\f12 U(t).
\end{align*}
Here we used the inequality
 $$A\ln(1+{U}^2)\leq A\ln(1+{(4A)}^2)+\f12 U,\quad \forall A, U\geq 0,$$
 which is trivial when $0\leq U\leq 4A$ and otherwise
follows from
\begin{align*}
A\ln(1+{U}^2)-A\ln(1+{(4A)}^2)&=\int_{4A}^U\frac{2Ax}{1+x^2}dx \leq \int_{4A}^U\frac{1}{2}dx\leq \f12 U.
\end{align*}
In particular for all $t_1,t_2\in[0,t]$, we have
\begin{align*}
&|V(t_1)|\leq|V(t_2)|+ 2\|K\|_0(2+\ln(1+(8\|K\|_0)^2))+ \f12 U(t).
\end{align*}
By the definition of the function $U(t)$, we therefore derive that
\begin{align*}
&U(t)\leq|V(s)|+ 2\|K\|_0(2+\ln(1+(8\|K\|_0)^2))+\f12 U(t),\quad \forall 0\leq s\leq t,
\end{align*}
from which we conclude that
\begin{align*}
U(t)&\leq2|V(s)|+ 4\|K\|_0(2+\ln(1+(8\|K\|_0)^2))\\&\leq C(|V(s)|+ \|K\|_0\ln(2+ \|K\|_0)),
\quad \forall\ s\in[0,t]
\end{align*}
for some constant $C$ independent of $T_*$, $t$, $s$. This implies that
\begin{align*}
|V(t)|+ \|K\|_0\ln(2+ \|K\|_0)&\leq U(t)+ \|K\|_0\ln(2+ \|K\|_0)
\leq C(|V(s)|+ \|K\|_0\ln(2+ \|K\|_0)),\\
|V(s)|+ \|K\|_0\ln(2+ \|K\|_0)&\leq (U(t)+ \|K\|_0\ln(2+ \|K\|_0))
\leq C(|V(t)|+ \|K\|_0\ln(2+ \|K\|_0)).
\end{align*}
We thus completed the proof for the Lemma.
\end{proof}
A direct consequence of this lemma is that the density distribution possesses the following pointwise decay estimate.
\begin{Prop}
\label{prop:bd4f}
For $K\in \mathcal{K}_{\Lambda}$ with $\Lambda > 2$, we have the decay estimate for the density distribution
$$f(t,x,v)\leq C(\Lambda)\varepsilon_0(1+|v|)^{-q}$$
for some constant $C(\Lambda) $ depending only on $\La$. Here the initial data $f_0=f(0, x, v)$ verifies the decay assumption of the main theorem.
\end{Prop}
\begin{proof}
In view of the definition for the characteristic, we show that
\begin{align*}
f(t,x,v)&=f_0(X(0,t,x,v),V(0,t,x,v))\\
&\leq \varepsilon_0(1+|V(0,t,x,v)|)^{-q}, \quad \forall t\geq 0, \quad x\in \mathbb{R}^3.
\end{align*}
The above Lemma \ref{lem:imbd:cha} with $T_*=t$ implies that
\begin{align*}
1+|v|&=1+|V(t,t,x,v)|\\
&\leq C(1+|V(0,t,x,v)|+\|K\|_0\ln(2+ \|K\|_0))\\
&\leq C(1+|V(0,t,x,v)|+\Lambda\ln(2+ \Lambda))\\
&\leq C\Lambda\ln\Lambda(1+|V(0,t,x,v)|).
\end{align*}
The decay estimate for the proposition then follows by setting
$C(\Lambda)=(C\Lambda\ln\Lambda)^q$.
\end{proof}

The previous Lemma \ref{lem:imbd:cha} can be viewed as estimate on the velocity along the characteristics. As our data are no longer compactly supported, we also need refined estimate on the space position along the characteristics.
\begin{Lem}
\label{lem:imbd:cha:x}
Let $K=(E, B)$ such that $\|K\|<+\infty$. For $(T_*, x, v)\in \mathbb{R}^{1+6}$, let $(X, V)$ be the associated characteristics. Then we have the following bound
\begin{align*}
&|X(t)-X(0)-t\hat{V}(t)|\leq C \La (\ln(1+t)+\ln(1+|X(0)|)),\quad \forall 0\leq t\leq T_*
\end{align*}
for some constant $C$. Here $ \Lambda=|V(0)|+(1+ \|K\|_0\ln(2+ \|K\|_0))^2$.
\end{Lem}
\begin{proof}
Similarly we define
\begin{align*}
&U(t)=\sup\{|V(s)|:0\leq s\leq t\},\quad \underline{U}(t)=\inf\{|V(s)|:0\leq s\leq t\}.
\end{align*}
Then we have  $\underline{U}(t)\leq |V(0)| $ and
\begin{align*}
 \sqrt{1-|\hat{V}(s)|^2}=(1+|V(s)|^2)^{-1/2}\leq(1+{\underline{U}}^2(t))^{-1/2}
\end{align*}
for $ 0\leq s\leq t.$ Note that
 \begin{align*}
 \frac{d\hat{V}}{ds}=J(s,X(s),\hat{V}(s))
\end{align*}
with
\begin{align}
\label{defJ}
 J(t,x,\hat{v})=\sqrt{1-|\hat{v}|^2}(E(t,x)+\hat{v}\times B(t,x)-\hat{v}\cdot E(t,x)\hat{v}).
\end{align}
In particular we have
\begin{align*}
|J(t,x,\hat{v})|\leq\sqrt{1-|\hat{v}|^2}(|E(t,x)|+|B(t,x)|),
\end{align*}
from which together with the definition of the norm $\|\cdot \|_0$ we derive that
\begin{align*}
 \Big|\frac{d\hat{V}}{ds}\Big| & \leq\sqrt{1-|\hat{V}|^2}(|E(s,X)|+|B(s,X)|)\\
 &\leq\frac{(1+{\underline{U}}^2(t))^{-1/2}\|K\|_0}{(\left|s-|X(s)|\right|+1)(s+|X(s)|+1)},\quad \forall 0\leq s\leq t.
\end{align*}
Now we compute that
\begin{align*}
\frac{d}{ds}(X-s\hat{V})=-s\frac{d\hat{V}}{ds},\quad \Big|\frac{d|X|}{ds}\Big|\leq \Big|\frac{dX}{ds}\Big|=|\hat{V}(s)|\leq\hat{U}(t).
\end{align*}
We then can show that
\begin{align*}
|X(t)-X(0)-t\hat{V}(t)|&\leq \int_0^ts\Big|\frac{d\hat{V}}{ds}\Big|ds\\
&\leq \|K\|_0\int_0^t\frac{s(1+{\underline{U}}^2(t))^{-1/2}}{(\left|s-|X(s)|\right|+1)(s+|X(s)|+1)}ds\\
&\leq \|K\|_0 \int_0^t\frac{(1+{\underline{U}}^2(t))^{-1/2}}{\left|s-|X(s)|\right|+1}ds\\
&\leq \|K\|_0 \frac{(1+{\underline{U}}^2(t))^{-1/2}}{1-\hat{U}(t)}\int_0^t\frac{\frac{d}{ds}(s-|X(s)|)}{\left|s-|X(s)|\right|+1}ds \\
&\leq \|K\|_0 {(1+{{U}}^2(t))(1+{\underline{U}}^2(t))^{-1/2}(\ln(1+t)+\ln(1+|X(0)|))}{}.
\end{align*}
Here we may note that $|X(t)|\leq |X(0)|+t$. 
Now in view of the previous Lemma
 \ref{lem:imbd:cha} by taking supreme on the left hand side of the inequality and infimum on the right,
  we conclude that
  \begin{align*}
  U(t)+ \|K\|_0\ln(2+ \|K\|_0)&\leq C(\underline{U}(t)+ \|K\|_0\ln(2+ \|K\|_0)).
\end{align*}
For the case when  $\underline{U}(t)\geq \|K\|_0\ln(2+ \|K\|_0) $, we in particular have that
$$U(t)\leq C\underline{U}(t)\leq C|V(0)|.$$
Therefore
\begin{align*}
&{(1+{{U}}^2(t))}{(1+{\underline{U}}^2(t))^{-1/2}}\leq C(1+{\underline{U}}^2(t))^{1/2}\leq C(1+|V(0)|)\leq C\Lambda.
\end{align*}
Otherwise if $\underline{U}(t)\leq \|K\|_0\ln(2+ \|K\|_0) $, then we have
$$U(t)\leq C\|K\|_0\ln(2+ \|K\|_0),$$
which then implies that
\begin{align*}
&{(1+{{U}}^2(t))}{(1+{\underline{U}}^2(t))^{-1/2}}\leq 1+{{U}}^2(t) \leq C(1+\|K\|_0\ln(2+ \|K\|_0))^2\leq C\Lambda.
\end{align*}
In particular we have shown that
\[
{(1+{{U}}^2(t))}{(1+{\underline{U}}^2(t))^{-1/2}}\leq C\Lambda .
\]
We therefore conclude that
\begin{align*}
&|X(t)-X(0)-t\hat{V}(t)|\leq C\Lambda(\ln(1+t)+\ln(1+|X(0)|)).
\end{align*}
This completes the proof.
\end{proof}

As discussed in the introduction, we deal with the issue of arbitrary large initial velocity by localizing the density function according to the velocity. However to insure the convergence of the approximate solutions, we need to control the diameter of the support of $f_{[k]}^{(n)}(t,x,\cdot)$.

\begin{Lem}
\label{lem:bd:cha:v}
Let $(X_1(s), V_1(s))$, $(X_2(s), V_2(s))$ be two characteristics associated to $K=(E, B)$ on $[0, t]$ such that
$$X_1(t)=X_2(t),\quad |X_i(0)|\leq R,\quad |V_i(0)|\leq R,\quad i=1, 2. $$
Then we have the decay estimate
\begin{align*}
&t|\hat{V}_1(t)-\hat{V}_2(t)|\leq C\Lambda(\ln(1+\|K\|_{1+\alpha}+\|K\|_{0}+R)+1)
\end{align*}
with $ \Lambda=R+(1+ \|K\|_0\ln(2+ \|K\|_0))^2$ for some constant $C$ depending only on $\alpha$.
\end{Lem}
\begin{proof}
In view of the previous Lemma \ref{lem:imbd:cha:x}, we first have
\begin{align*}
 &|X_i(s)-X_i(0)-s\hat{V}_i(s)|\leq C\Lambda(\ln(1+s)+\ln(1+R)),\quad i=1, 2.
\end{align*}
In particular we conclude that
\begin{align*}
\numberthis\label{s1}
 &|X_1(s)-X_2(s)-s(\hat{V}_1(s)-\hat{V}_2(s))|
 \leq C\Lambda(\ln(1+s)+\ln(1+R)+1)
\end{align*}
as $\La\geq R$. Using the characteristic equation, we compute that
\begin{align*}
 &X_1(s)-X_2(s)-s(\hat{V}_1(s)-\hat{V}_2(s))=s^2\frac{d}{ds}\frac{X_1-X_2}{s}.
\end{align*}
Hence from the previous inequality \eqref{s1} we derive that
\begin{align*}
&\left|\frac{d}{ds}\frac{X_1-X_2}{s}\right|\leq \frac{C\Lambda}{s^2}(\ln(1+s)+\ln(1+R)+1).
\end{align*}
By the assumption $ X_1(t)=X_2(t)$, the above inequality leads to
\begin{align*}
\left|\frac{X_1(s)-X_2(s)}{s}\right|&\leq \int_s^t\frac{C\Lambda}{\tau^2}(\ln(1+\tau)+\ln(1+R)+1)d\tau\\&\leq \frac{C\Lambda}{s}(\ln(1+s)+\ln(1+R)+1).
\end{align*}
In view of estimate \eqref{s1}, we then can bound that
\begin{align}
\label{s2}
&|X_1(s)-X_2(s)|+s|\hat{V}_1(s)-\hat{V}_2(s)|\leq  C\Lambda(\ln(1+s)+\ln(1+R)+1).
\end{align}
Now define
\begin{align*}
&U=\sup\{|V_i(s)|:0\leq s\leq t,\ i\in\{1,2\}\},\quad \hat{U}=U/\sqrt{1+U^2}.
\end{align*}
In particular we have
$$|V_i(s)|\leq U,\quad |\hat{V}_i(s)|\leq \hat{U}< 1,\quad 1-\hat{U}\geq (1-\hat{U}^2)/2=1/(2U^2+2).$$
Moreover using the equation for $X_i$, we also have
\begin{align*}
&|X_i(s)|\leq |X_i(0)|+s\hat{U}\leq R+s(1-1/(2U^2+2)),\quad \forall\ s\in[0,t].
\end{align*}
Recall the definition of $J$ in \eqref{defJ} in the proof of Lemma \ref{lem:imbd:cha:x}.
Then for $$|v|\leq U,\quad t-|x|+R\geq t/(2U^2+2),\quad t\geq 4R(U^2+1),$$
 we can estimate that
\begin{align*}
 |\nabla_{\hat{v}}J|&\leq C(1-|\hat{v}|^2)^{-1/2}\|K\|_{0}(\left|t-|x|\right|+1)^{-1}(t+|x|+1)^{-1}\\&\leq C\sqrt{1+|v|^2}\|K\|_{0}(t/(4U^2+4)+1)^{-1}(t+1)^{-1}\\&\leq C\sqrt{1+|v|^2}\|K\|_{0}(U^2+1)t^{-2}\\
 &\leq C\|K\|_{0}(U^2+1)^{3/2}t^{-2}.
\end{align*}
Now assume $ |y|\geq |x|$ verifying the same assumption as above (hence we also have $t-|y|\geq t/(4U^2+4))$. Then we can show that
\begin{align*}
|J(t,x,\hat{v})-J(t,y,\hat{v})| &\leq \sqrt{1-|\hat{v}|^2}(|E(t,x)-E(t,y)|+|B(t,x)-B(t,y)|)\\
 &\leq (t-|y|+1)^{-1-\alpha}(t+1)^{-1}(|x-y|+1)\|K\|_{1+\alpha}\\
 &\leq (t/(4U^2+4)+1)^{-1-\alpha}(t+1)^{-1}(|x-y|+1)\|K\|_{1+\alpha}\\
 &\leq C(U^2+1)^{1+\alpha}t^{-2-\alpha}(|x-y|+1)\|K\|_{1+\alpha}.
\end{align*}
By applying the mean value theorem, for the case when $4R(U^2+1)\leq s\leq t$
it follows that
\begin{align*}
&|J(s,X_1(s),\hat{V}_1(s))-J(s,X_2(s),\hat{V}_2(s))|\\
&\leq |J(s,X_1(s),\hat{V}_2(s))-J(s,X_2(s),\hat{V}_2(s))|+|J(s,X_1(s),\hat{V}_1(s))-J(s,X_1(s),\hat{V}_2(s))|\\
&\leq C\|K\|_{1+\alpha}(U^2+1)^{1+\alpha}s^{-2-\alpha}(|X_1(s)-X_2(s)|+1)
+C\|K\|_{0}(U^2+1)^{3/2}s^{-2}|\hat{V}_1(s)-\hat{V}_2(s)|\\
& \leq C(U^2+1)^{3/2}\Lambda(s^{-2-\alpha}\|K\|_{1+\alpha}+s^{-3}\|K\|_{0})(\ln(1+s)+\ln(1+R)+1).
\end{align*}
Here the last step follows from estimate \eqref{s2} and the assumption $\Lambda\geq 1$, $0<\alpha<1/2$. Now define
$$T_0=(U^2+1)^{\frac{3}{2\alpha}}(\|K\|_{1+\alpha}^{\frac{1}{\a}}+\|K\|_{0}+4R).$$
In particular $T_0\geq 4R(U^2+1)>R $. If $t\leq T_0$, by \eqref{s2} 
we have
\begin{align*}
&t|\hat{V}_1(t)-\hat{V}_2(t)|\leq C\Lambda(\ln(1+t)+\ln(1+R)+1)\leq C\Lambda(\ln(1+T_0)+1).
\end{align*}
Otherwise if $t\geq T_0$, note that
\begin{align*}
&\frac{d}{ds}(X_1(s)-X_2(s)-s(\hat{V}_1(s)-\hat{V}_2(s)))=-s\frac{d(\hat{V}_1-\hat{V}_2)}{ds}=-s(J(s,X_1(s),\hat{V}_1(s))-J(s,X_2(s),\hat{V}_2(s))),
\end{align*}
from which we conclude that
\begin{align*}
 &|X_1(s)-X_2(s)-s(\hat{V}_1(s)-\hat{V}_2(s))|\big|_{s=T_0}^{s=t}\\
 & \leq \int_{T_0}^ts|J(s,X_1(s),\hat{V}_1(s))-J(s,X_2(s),\hat{V}_2(s))|ds\\
 & \leq \int_{T_0}^tC(U^2+1)^{3/2}\Lambda(s^{-1-\alpha}\|K\|_{1+\alpha}+s^{-2}\|K\|_{0})(\ln(1+s)+\ln(1+R)+1)ds\\
 & \leq C(U^2+1)^{3/2}\Lambda(T_0^{-\alpha}\|K\|_{1+\alpha}+T_0^{-1}\|K\|_{0})(\ln(1+T_0)+\ln(1+R)+1)ds\\
 &\leq C\Lambda(\ln(1+T_0)+\ln(1+R)+1).
\end{align*}
From estimate \eqref{s1} (as in this case $t\geq T_0$), we derive that
\begin{align*}
&|X_1(s)-X_2(s)-s(\hat{V}_1(s)-\hat{V}_2(s))|\big|_{s=T_0} \leq C\Lambda(\ln(1+T_0)+\ln(1+R)+1).
\end{align*}
Since $X_1(t)=X_2(t)$, we therefore conclude from the previous estimate that
\begin{align*}
 &t|\hat{V}_1(t)-\hat{V}_2(t)|\leq C\Lambda(\ln(1+T_0)+\ln(1+R)+1)\leq C\Lambda(\ln(1+T_0)+1).
\end{align*}
Therefore in any case we always have
\begin{align*}
t|\hat{V}_1(t)-\hat{V}_2(t)| & \leq C\Lambda(\ln(1+T_0)+1)\\
&\leq C\Lambda(\ln(1+U)+\ln(1+\|K\|_{1+\alpha}+\|K\|_{0}+R)+1)\\
& \leq C\Lambda(\ln(1+\|K\|_{1+\alpha}+\|K\|_{0}+R)+1)
\end{align*}
for some constant $C$ depending only on $\a$. Here we used the fact that
$$U\leq C(\max(|X_1(0)|,|X_2(0)|)+ \|K\|_0\ln(2+ \|K\|_0))\leq C(R+ \|K\|_0^2+1)$$
in view of Lemma \ref{lem:imbd:cha}. This completes the proof.
\end{proof}

\section{Decomposition of the density distribution}
\label{sec:localization}
Since there is no restriction on the support of the initial data, the initial velocity can be arbitrarily large. Under the frame work of the existing approach to study RVM, we localize the density function in the following way: Fix an even smooth function $ \widetilde{\psi}:\R\to[0,1]$, which is supported on $[-2,2]$ and equals to $1$ on $[-1,1]$. Define
\[
\psi_0(x,v):=\widetilde{\psi}(|(x,v)|), \quad \psi_k(x,v):=\widetilde{\psi}(|(x,v)|/2^k)-\widetilde{\psi}(|(x,v)|/2^{k-1})
\]
for positive integers $k$
with the natural norm $|(x,v)|=(|x|^2+|v|^2)^{1/2}$. Then define $$f_{0,k}=f_0\psi_{k-1},\quad k\geq 1.$$
In particular $f_{0, k}$ is supported on $\{(x, v)| |x|^2+|v|^2\leq 2^{2k+2}\}$ verifying the property
\begin{align*}
&\sum_{k=0}^{+\infty}\psi_k(x,v)=1,\quad \sum_{k=1}^{+\infty}f_{0,k}(x,v)=f_0(x,v),\\
& \|f_{0,k}\|_0:=\sup_{x,v}|f_{0,k}(x,v)|\leq 2^{(2-k)q}\varepsilon_0.
\end{align*}
Now let $f_{[k]}=f_{[k]}^{(n)}$ (the $n$-th component of the iteration sequence, for simplicity we drop the dependence on $n$) be the solution of the linear Vlasov equation
\begin{align*}
&\partial_tf_{[k]}+\hat{v}\cdot\nabla_xf_{[k]}+(E+\hat{v}\times B)\cdot\nabla_vf_{[k]}=0,\ f_{[k]}(0,x,v)=f_{0,k}(x,v).
\end{align*}
If $(X, V)$ is the associated characteristics, then we have
\begin{align*}
&f_{[k]}(t,x,v)=f_{0,k}(X(0,t,x,v),V(0,t,x,v)),\quad f=\sum_{k=1}^{+\infty}f_{[k]} .
\end{align*}
Using the characteristic equation for $X$, it is obvious that
\[
|X(s,t,x,v)-x|\leq |s-t|.
\]
As $f_{0,k}=0$ when $|x|\geq 2^{k+1}$,  we  therefore conclude that
$$f_{[k]}(t,x,v)=0,\quad \forall |x|\geq t+2^{k+1}.$$
Moreover in view of Lemma \ref{lem:imbd:cha},  we also have
 \begin{align*}
 &|v|\leq C(|V(0,t,x,v)|+ \|K\|_0\ln(2+ \|K\|_0)).
\end{align*}
The discussion leads to the following bound for the velocity support of $f_{[k]}$.
\begin{Lem}
\label{lem:bd:support:fk:v}
If $f_{[k]}(t,x,v)\neq 0 $ then $$|v|\leq C(2^k+ \|K\|_0\ln(2+ \|K\|_0))$$
for some constant $C$ independent of $k$. Here $K=(E, B)$.
\end{Lem}
\begin{proof}
The lemma follows by the fact that  $f_{[k]}(t,x,v)\neq 0 $ occurs only when  $$|V(0,t,x,v)|\leq 2^{k+1},\quad  |X(0,t,x,v)|\leq 2^{k+1}.$$
\end{proof}

To close the argument, we also need an improved estimate for the charge density.
\begin{Prop}
\label{prop:bd:averaged:f}
For positive integers $k$, denote
\begin{align*}
\Lambda_{k,i}&=2^{k}+ (\|K\|_0\ln(2+ \|K\|_0))^{i},\quad i=1,2,\\
\Lambda_{k,3}&=\ln(1+\|K\|_{1+\alpha}+\|K\|_{0})+k+1.
\end{align*}
Then for all $t\geq 0,\ x\in\R^3$ we have
\begin{align*}
\int_{\R^3}f_{[k]}(t,x,v)dv\leq C\|f_{0,k}\|_0(2^k+t+|x|)^{-3}\Lambda_{k,1}^5\Lambda_{k,2}^3\Lambda_{k,3}^3.
\end{align*}
\end{Prop}
The proof relies on the following result.
\begin{Lem}
\label{lem:fromSch}
For $P\geq 1$, $\delta>0$ and $v\in\mathbb{R}^3$ such that $|v|\leq P$, define the set
 \begin{align*}
 S=\{w| w\in\mathbb{R}^3,\quad  |w|\leq P\quad |\hat{v}-\hat{w}|\leq \delta\}.
\end{align*}
Then the Lebesgue measure of $S$ can be bounded as follows
\[
\mu(S)\leq CP^5\delta^3
\]
 for some constant $C$ independent of $P$ or $\delta$.
 \end{Lem}
 \begin{proof}
 See Lemma 1.4 in \cite{Schaeffer:VM:notcompat:p}.
 \end{proof}
We now can prove Proposition \ref{prop:bd:averaged:f}.
For fixed $t\geq 0,\ x\in\R^3,$ define
\begin{align*}
&S_0=\{w| f_{[k]}(t,x,w)\neq 0\}.
\end{align*}
The case when $S_0$ is empty holds automatically. In the sequel, let's fix $v\in S_0$. Denote
 $$P=\sup\{|w||w\in S_0\}.$$
 In particular $|v|\leq P.$ From the above Lemma \ref{lem:bd:support:fk:v}, we derive that
 $$P\leq C(2^k+ \|K\|_0\ln(2+ \|K\|_0))=C\Lambda_{k,1}.$$
 By the discussion before Lemma \ref{lem:bd:support:fk:v}, for any $w\in S_0$, we have
 \begin{align*}
 &\max(|X(0,t,x,v)|, \quad |X(0,t,x,w)|,\quad |V(0,t,x,v)|,\quad |V(0,t,x,w)|)\leq 2^{k+1}.
\end{align*}
Then in view of Lemma \ref{lem:bd:cha:v} with $R=2^{k+1}$, we conclude that
\begin{align*}
t|\hat{v}-\hat{w}|&\leq C\Lambda(\ln(1+\|K\|_{1+\alpha}+\|K\|_{0}+2^k)+1)
\leq C\Lambda_{k,2}\Lambda_{k,3}.
\end{align*}
Here  $ \Lambda=2^{k+1}+(1+\|K\|_0\ln(2+ \|K\|_0))^2$. This means that for $t\geq 2^k$, the set $S_0$ is a subset of
\begin{align*}
&S=\{w|\ |w|\leq P\ \text{and}\ |\hat{v}-\hat{w}|\leq C t^{-1} \Lambda_{k,2}\Lambda_{k,3}\}.
\end{align*}
Then using Lemma \ref{lem:fromSch}, for $t\geq 2^k$, we conclude that
\begin{align*}
&\mu(S_0)\leq CP^5(Ct^{-1}\Lambda_{k,2}\Lambda_{k,3})^3\leq Ct^{-3}\Lambda_{k,1}^5\Lambda_{k,2}^3\Lambda_{k,3}^3\leq C(2^k+t)^{-3}\Lambda_{k,1}^5\Lambda_{k,2}^3\Lambda_{k,3}^3.
\end{align*}
Otherwise if $t\leq 2^k$, since $\Lambda_{k,1}\geq 1$, $\Lambda_{k,2}\geq 2^k\geq t$ and $\Lambda_{k,3}\geq 1 $, we in particular have that
 $$\mu(S_0)\leq \mu(\{w|\ |w|\leq P\})\leq CP^3\leq C\Lambda_{k,1}^3\leq C (2^k+t)^{-3}\Lambda_{k,1}^5\Lambda_{k,2}^3\Lambda_{k,3}^3. $$
 This bound for the measure of the set $S_0$ then leads to
 \begin{align*}
 \int_{\R^3}f_{[k]}(t,x,v)dv\leq\|f_{0,k}\|_0\mu(S_0)\leq C\|f_{0,k}\|_0(2^k+t)^{-3}\Lambda_{k,1}^5\Lambda_{k,2}^3\Lambda_{k,3}^3.
\end{align*}
This proves Proposition \ref{prop:bd:averaged:f}.

 \bigskip

To control the Maxwell field, we also need an estimate on the weighted charge on backward cones, which relies on the following conservation of momentum:
  \begin{align*}
  \numberthis\label{1.3}
  &\int_{|y-x|\leq t}\int_{\R^3}(1+\hat{v}\cdot\omega)f_{[k]}(t-|y-x|,y,v){dvdy} = \int_{|y-x|\leq t}\int_{\R^3}f_{[k]}(0,y,v){dvdy}.
\end{align*}
Here $\omega=\frac{y-x}{|x-y|}$.
In fact we can write the Vlasov equation as
\begin{align*} &\partial_tf_{[k]}+\nabla_x\cdot[\hat{v}f_{[k]}]+\nabla_v\cdot[(E+\hat{v}\times B)f_{[k]}]=0.
\end{align*}
Integrate this equation on the domain $\{(s,y,v)|0\leq s\leq t-|y-x|,\ y,v\in\R^3\}$. The above conservation then follows by using Stokes formula.
This conservation of momentum is used to prove the following decay estimate for the weighted charge.
\begin{Lem}
\label{lem:decay:fk}
Let $\Lambda_{k,i}$ be constants defined in Proposition \ref{prop:bd:averaged:f}. Then for all  $0\leq p\leq 2$, $ \forall t\geq 0$, $x\in\mathbb{R}^3$  and positive integer $k$, we have the decay estimate
\begin{align*}
&\int_{|y-x|\leq t}\int_{\R^3}f_{[k]}(t-|y-x|,y,v)\frac{dvdy}{|y-x|^p} \leq C\|f_{0,k}\|_0(2^k+t)^{-p}2^{(6-2p)k}\Lambda_{k,1}^{2+p}\Lambda_{k,2}^p\Lambda_{k,3}^p.
\end{align*}
\end{Lem}
\begin{proof}
First by the definition of $f_{[k]}(0, y, v)$, we can show that
\begin{align*}
\int_{|y-x|\leq t}\int_{\R^3}f_{[k]}(0,y,v){dvdy}&=\int_{|y-x|\leq t,|y|\leq 2^{k+1}}\int_{|v|\leq 2^k}f_{0,k}(y,v){dvdy}\\
&\leq  C(2^k\min(2^k,t))^3\|f_{0,k}\|_0\\
&\leq C2^{6k}\|f_{0,k}\|_0,
\end{align*}
For $v\in\mathbb{R}^3$ such that $f_{[k]}(t, y, v)\neq 0$, by using Lemma \ref{lem:bd:support:fk:v}, we have the lower bound
\begin{align*}
1+\hat{v}\cdot\omega\geq \f12(1+|v|^2)^{-1}\geq C^{-1} \Lambda_{k, 1}^{-2}.
\end{align*}
Then from the above conservation of momentum \eqref{1.3}, we can show that
\begin{align*}
\int_{|y-x|\leq t}\int_{\R^3}f_{[k]}(t-|y-x|,y,v){dvdy}
& \leq C\Lambda_{k,1}^2\int_{|y-x|\leq t}\int_{\R^3}(1+\hat{v}\cdot\omega)f_{[k]}(t-|y-x|,y,v){dvdy}\\
& \leq  C\Lambda_{k,1}^2(2^k\min(2^k,t))^3\|f_{0,k}\|_0\\
& \leq  C\Lambda_{k,1}^22^{6k}\|f_{0,k}\|_0.
\end{align*}
In particular, the Lemma holds for $p=0.$

Now for any $ \delta\in(0,1/2]$ and $p\in(0,2]$, we can estimate that
\begin{align*}
\int_{\delta t\leq |y-x|\leq t}\int_{\R^3}f_{[k]}(t-|y-x|,y,v)\frac{dvdy}{|y-x|^p}
& \leq \frac{1}{(\delta t)^p}\int_{|y-x|\leq t}\int_{\R^3}f_{[k]}(t-|y-x|,y,v){dvdy}\\
& \leq \frac{C}{(\delta t)^p}\Lambda_{k,1}^2(2^k\min(2^k,t))^3\|f_{0,k}\|_0 \\
&\leq \frac{C\Lambda_{k,1}^22^{6k}}{\delta^p(2^k+t)^p}\|f_{0,k}\|_0.
\end{align*}
For the integral on the region close to the light cone, in view of the pointwise estimate of Proposition \ref{prop:bd:averaged:f}, we can bound that
\begin{align*}
&\int_{|y-x|\leq \delta t}\int_{\R^3}f_{[k]}(t-|y-x|,y,v)\frac{dvdy}{|y-x|^p}\\
& \leq C\Lambda_{k,1}^5\Lambda_{k,2}^3\Lambda_{k,3}^3\int_{|y-x|\leq \delta t}(2^k+t-|y-x|)^{-3}\|f_{0,k}\|_0\frac{dy}{|y-x|^p}\\
& \leq C\Lambda_{k,1}^5\Lambda_{k,2}^3\Lambda_{k,3}^3(\delta t)^{3-p}(2^k+t-\delta t)^{-3}\|f_{0,k}\|_0\\
& \leq C\Lambda_{k,1}^5\Lambda_{k,2}^3\Lambda_{k,3}^3\delta^{3-p}(2^k+t)^{-p}\|f_{0,k}\|_0.
\end{align*}
Therefore for all $\delta\in(0, \frac{1}{2}]$ we have shown that
\begin{align*}
&\int_{|y-x|\leq  t}\int_{\R^3}f_{[k]}(t-|y-x|,y,v)\frac{dvdy}{|y-x|^p} \leq C\Lambda_{k,1}^2(2^k+t)^{-p}\delta^{-p}(2^{6k}+\Lambda_{k,1}^3\Lambda_{k,2}^3\Lambda_{k,3}^3\delta^{3})\|f_{0,k}\|_0.
\end{align*}
As $\Lambda_{k,1}\geq 2^k$, $\Lambda_{k,2}\geq 2^k$, $\Lambda_{k,3}\geq k+1\geq 2, $ we have $\Lambda_{k,1}\Lambda_{k,2}\Lambda_{k,3}\geq 2^{2k+1}.$ The Lemma then follows by taking
$\delta=2^{2k}/(\Lambda_{k,1}\Lambda_{k,2}\Lambda_{k,3})$.
\end{proof}

\section{Bound for the Maxwell field}

We construct the solution iteratively: starting with a given Maxwell field, we solve the Vlasov equation via characteristics with estimates derived in the previous section. Then with the density distribution, we investigate the linear Maxwell equation. For symmetry, we will mainly carry out the detailed analysis on the electric field $E$, which heavily relies on the following representation formula for $E^*$ of linear Maxwell equation
\begin{equation}
\label{eq:def:electric}
\begin{split}
E^* &=E_z^*+E_T^*+E_S^*,\\
E_T^*(t,x)&=-\int_{|y-x|\leq t}\int_{\R^3}\frac{(\omega+\hat{v})(1-|\hat{v}|^2)}{(1+\hat{v}\cdot\omega)^2}f(t-|y-x|,y,v)\frac{dvdy}{|y-x|^2},\\
 E_S^*(t,x)&=-\int_{|y-x|\leq t}\int_{\R^3}\nabla_v\left[\frac{\omega+\hat{v}}{1+\hat{v}\cdot\omega}\right]\cdot(E+\hat{v}\times B)f|_{(t-|y-x|,y,v)}\frac{dvdy}{|y-x|},\\
 E_z^*(t,x)&=\mathcal{E}(t,x)-\frac{1}{t}\int_{|y-x|= t}\int_{\R^3}\frac{\omega+\hat{v}}{1+\hat{v}\cdot\omega}f_0(y,v)dvdS_y.
 \end{split}
\end{equation}
Here $ \omega=(y-x)/|y-x|$ and $\mathcal{E}(t,x) $ is defined in \eqref{defE}. The above formulae could be found, for example, in Theorem 3 of \cite{Glassey:VM:singularity:86}.

Note that
$E_z^*$ depends only on the initial data. We first bound the linear evolution $\mathcal{E}(t,x) $, which relies on the following integration lemma.
\begin{Lem}
\label{lem:bd:linear:Maxwell}
Assume that $h\in C({\R^3})$ such that $$|h(x)|\leq K_0(1+|x|)^{-k}$$ for some constant $K_0$. Then for $t\geq 0,\ x\in\R^3$
we have
\begin{align*}
\left|\int_{|x-y|=t}h(y)dS_y\right|\leq \left\{\begin{array}{ll}8\pi K_0t^2(1+t+|x|)^{-1}(1+|t-|x||)^{-1},&  k=2,\\4\pi K_0t(1+t+|x|)^{-1}(1+|t-|x||)^{-k+2},& k\geq 3.
\end{array}
\right.
\end{align*}
\end{Lem}
\begin{proof}
By direct computation or from the proof of Lemma 5.7 in \cite{Rein90:VM}, we can show that
\begin{align*}
&\left|\int_{|x-y|=t}h(y)dS_y\right|\leq 2\pi K_0 t r^{-1} \int_{|t-r|}^{t+r}\frac{\lambda d\lambda}{(1+\lambda)^k},\quad r=|x|.
\end{align*}
Denote $a=t+r$, $b=|t-r|$.  For the case when $k=2$, we can compute that
\begin{align*}
\int_{a}^{b}\frac{\lambda d\lambda}{(1+\lambda)^2}&=\int_{b}^{a}\frac{ d\lambda}{1+\lambda}-\int_{b}^{a}\frac{ d\lambda}{(1+\lambda)^2}  \leq \frac{ a-b}{1+b}-\frac{ a-b}{(1+a)(1+b)}
 \leq \frac{ 4rt}{(1+a)(1+b)}.
\end{align*}
Hence the Lemma holds for the case when $k=2$.

When
$k\geq 3$, we can bound that
\begin{align*}
\int_{a}^{b}\frac{\lambda d\lambda}{(1+\lambda)^k} &\leq\int_{b}^{a}\frac{ d\lambda}{(1+\lambda)^{k-1}}\\
&=\frac{1}{k-2} (1+a)^{-k+2}(1+b)^{-k+2} ((1+a)^{k-2}-(1+b)^{k-2}) \\
&\leq (1+a)^{-k+2}(1+b)^{-k+2} (a-b) (1+a)^{k-3}\\
&\leq 2r (1+a)^{-1}(1+b)^{-k+2}.
\end{align*}
This shows that the Lemma holds when $k\geq 3$.
\end{proof}

With this lemma, we then can control the linear evolution of the Maxwell field which. 
\begin{Prop}
\label{prop:bd4:lMaxwell}
The linear Maxwell field $(\mathcal{E}, \mathcal{B})$ verifies the following bound
$$\|(\mathcal{E},\mathcal{B})\|\leq CM$$
for some constant $C>0$. Here the norm $\|\cdot \|$ is the $K$-norm defined in \eqref{eq:def4:Knorm}.
\end{Prop}
\begin{proof}
For symmetry we only prove the bound for the linear electric field $ \mathcal{E}(t,x)$
, which could be decomposed into three parts
\begin{align*}
&\mathrm{I}_1=\frac{1}{4\pi t^2}\int_{|y-x|= t}E_0(y)dS_y,\quad
\mathrm{I}_3=\frac{1}{4\pi t}\int_{|y-x|= t}\text{curl}B_0(y)dS_y,\\
& \mathrm{I}_2=\frac{1}{4\pi t^2}\int_{|y-x|= t}((y-x)\cdot\nabla)E_0(y)dS_y.
\end{align*}
By the assumption on the initial Maxwell field and the previous Lemma \ref{lem:bd:linear:Maxwell}, we conclude that
\begin{align*}
&|\mathrm{I}_1 |+|\mathrm{I}_2|+|\mathrm{I}_3|\leq 6 M(1+t+|x|)^{-1}(1+|t-|x||)^{-1}.
\end{align*}
This implies that
\begin{align*}
\|(\mathcal{E},0)\|_0\leq CM.
\end{align*}
To control the mixed Lipschitz norm $\|\cdot\|_{1+\a}$ norm, we appeal to the derivative of the Maxwell field. Note that
\begin{align*}
&\nabla \mathrm{I}_1=\frac{1}{4\pi t^2}\int_{|y-x|= t}\nabla E_0(y)dS_y,\quad  \nabla \mathrm{I}_3=\frac{1}{4\pi t}\int_{|y-x|= t}\nabla \text{curl}B_0(y)dS_y,\\
& \nabla\mathrm{I}_2=\frac{1}{4\pi t^2}\int_{|y-x|= t}((y-x)\cdot\nabla)\nabla E_0(y)dS_y.
\end{align*}
Again the assumption on the initial Maxwell field together with
Lemma \ref{lem:bd:linear:Maxwell} shows that
\begin{align*}
&|\nabla_x\mathcal{E}(t,x)|\leq |\nabla \mathrm{I}_1 |+|\nabla \mathrm{I}_2 |+|\nabla \mathrm{I}_3 |\leq CM t^{-1}(1+|t-|x||)^{-2}.
\end{align*}
Now if $|x|\leq |y|\leq t$ and $t\geq 1,$ by applying the mean value theorem to $\mathcal{E}$, it follows that
\begin{align*}
|\mathcal{E}(t,x)-\mathcal{E}(t,y)|&\leq |x-y|\sup_{s\in[0,1]}|\nabla_x\mathcal{E}(t,sx+(1-s)y)|\\ &\leq CM|x-y|t^{-1}(1+|t-|y||)^{-2}\\
&\leq CM(|x-y|+1)(t+1)^{-1}(t-|y|+1)^{-1-\alpha}.
\end{align*}
Otherwise if $|x|\leq |y|\leq t\leq 1$ then
\begin{align*}
&|\mathcal{E}(t,x)-\mathcal{E}(t,y)|\leq CM \leq CM(|x-y|+1)(t+1)^{-1}(t-|y|+1)^{-1-\alpha}.\end{align*}
Hence by definition, we have
\begin{align*}
\|(\mathcal{E},0)\|_{1+\alpha}\leq CM.
\end{align*}
Combining with the above bound for $\|(\mathcal{E}, 0)\|_{0}$, we thus have shown that $$\|(\mathcal{E},0)\|\leq CM.$$
The estimates for the magnetic field $\mathcal{B}$ can be obtained in a similar way. 
\end{proof}
Once we have bound for the pure linear Maxwell field, we now can control the full linear Maxwell field $E_z^{*}$, relying only on the initial data.
\begin{Prop}
\label{prop:bd:Ezstar}
The full linear electric field $E_z^{*}$ defined in \eqref{eq:def:electric} verifies the following bound $$\|E_z^*-\mathcal{E}\|\leq C\varepsilon_0.$$
\end{Prop}
\begin{proof}
Recall the representation formula for $E_z^*$ in \eqref{eq:def:electric}. Note that
\begin{align*}
\left|\frac{\omega+\hat{v}}{1+\hat{v}\cdot\omega}\right|\leq\frac{\sqrt{2}}{(1+\hat{v}\cdot\omega)^{1/2}}\leq 2\sqrt{1+|v|^2}.
\end{align*}
Therefore we derive that
\begin{align}
\label{Ez1}
|E_z^*(t,x)-\mathcal{E}(t,x)|\leq\frac{1}{t}\int_{|y-x|= t}\int_{\R^3}2\sqrt{1+|v|^2}f_0(y,v)dvdS_y.
\end{align}
By the assumption on the initial density distribution $f_0$, for $y\in\R^3$
 we have
 \begin{align*}
 \int_{\R^3}2\sqrt{1+|v|^2}f_0(y,v)dv&\leq\int_{\R^3}2(1+|v|)\varepsilon_0(1+|y|+|v|)^{-q}dv
  \leq C\varepsilon_0(1+|y|)^{-4}.
\end{align*}
Here recall that $q>9$. Then in view of
Lemma \ref{lem:bd:linear:Maxwell} applied to \eqref{Ez1} with $k=4$, we derive that
\begin{align*}
|E_z^*(t,x)-\mathcal{E}(t,x)|&\leq  C\varepsilon_0(1+t+|x|)^{-1}(1+|t-|x||)^{-2}\\&\leq  C\varepsilon_0(1+t+|x|)^{-1}(1+|t-|x||)^{-1},
\end{align*}
which implies $$\|(E_z^*-\mathcal{E},0)\|_0\leq  C\varepsilon_0. $$
For the mixed Lipschitz norm, for $|x|\leq |y|\leq t$,
 we can directly bound that
 \begin{align*}
 &|(E_z^*(t,x)-\mathcal{E}(t,x))-(E_z^*(t,y)-\mathcal{E}(t,y))|\\
 & \leq  |E_z^*(t,x)-\mathcal{E}(t,x)|+|E_z^*(t,y)-\mathcal{E}(t,y)|\\
 & \leq C\varepsilon_0(1+t+|x|)^{-1}(1+|t-|x||)^{-2}+C\varepsilon_0(1+t+|y|)^{-1}(1+|t-|y||)^{-2}\\ 
 &\leq C\varepsilon_0(|x-y|+1)(1+t)^{-1}(1+|t-|y||)^{-1-\alpha},
 \end{align*}
 which implies $\|(E_z-\mathcal{E},0)\|_{1+\alpha}\leq  C\varepsilon_0. $ Thus $\|(E_z-\mathcal{E},0)\|\leq  C\varepsilon_0. $
 \end{proof}
We next estimate the other part of the full electric field. First we write that
\begin{align}
\label{ETSk}
E_T^*(t,x)=\sum_{k=1}^{+\infty}E_{T,k}^*(t,x),\quad E_S^*(t,x)=\sum_{k=1}^{+\infty}E_{S,k}^*(t,x)
\end{align}
with
\begin{align*}
E_{T,k}^*(t,x)&=-\int_{|y-x|\leq t}\int_{\R^3}\frac{(\omega+\hat{v})(1-|\hat{v}|^2)}{(1+\hat{v}\cdot\omega)^2}f_{[k]}(t-|y-x|,y,v)\frac{dvdy}{|y-x|^2},\\
 E_{S,k}^*(t,x)&=-\int_{|y-x|\leq t}\int_{\R^3}\nabla_v\left[\frac{\omega+\hat{v}}{1+\hat{v}\cdot\omega}\right]\cdot(E+\hat{v}\times B)f_{[k]}|_{(t-|y-x|,y,v)}\frac{dvdy}{|y-x|}.
\end{align*}
Recall that $f_{[k]}(t,x,v)=0 $ for $|x|\geq t+2^{k+1}$. In particular
 $$E_{T,k}^*(t,x)=E_{S,k}^*(t,x)=0, \quad  \forall |x|\geq t+2^{k+1},\quad \forall k\geq 1. $$
We now can prove the decay estimates for the perturbation part of the Maxwell field.
\begin{Prop}
\label{prop:decay:Ek:pt}
For positive integer $k$, let $\Lambda_{k,i}$ be constants defined in Proposition \ref{prop:bd:averaged:f}.  Then we have the following pointwise decay estimates
\begin{align*}
|E_{T,k}^*(t,x)|&\leq C2^{2k}\Lambda_{k,1}^5\Lambda_{k,2}^2\Lambda_{k,3}^2(|x|+t+1)^{-2}\|f_{0,k}\|_0,
\\|E_{S,k}^*(t,x)|&\leq \frac{C2^{4k}\Lambda_{k,1}^4\Lambda_{k,2}\Lambda_{k,3}\|K\|_0\|f_{0,k}\|_0}{(|t-|x||+1)(|x|+t+1)},\quad \forall t\geq 0, \quad x\in \R^3.
\end{align*}
\end{Prop}
\begin{proof}
For $E_{T, k}^*$, one can show that
\begin{align}
\label{om1}
\frac{|\omega+\hat{v}|(1-|\hat{v}|^2)}{(1+\hat{v}\cdot\omega)^2}\leq \frac{3\sqrt{3}}{4}\sqrt{1+|v|^2}.
\end{align}
For details regarding this inequality, we refer to equation (1.27) in \cite{Schaeffer:VM:notcompat:p}. Then by using the bound for the velocity in Lemma \ref{lem:bd:support:fk:v}, we can show that
\begin{align*}
|E_{T,k}^*(t,x)|&\leq C\int_{|y-x|\leq t}\int_{\R^3}\sqrt{1+|v|^2}f_{[k]}(t-|y-x|,y,v)\frac{dvdy}{|y-x|^2}\\&\leq C\Lambda_{k,1}\int_{|y-x|\leq t}\int_{\R^3}f_{[k]}(t-|y-x|,y,v)\frac{dvdy}{|y-x|^2}.
\end{align*}
Then by Lemma \ref{lem:decay:fk} with $p=2$, we derive that
\begin{align*}
|E_{T,k}^*(t,x)|&\leq C2^{2k}\Lambda_{k,1}^5\Lambda_{k,2}^2\Lambda_{k,3}^2(2^k+t)^{-2}\|f_{0,k}\|_0.
\end{align*}
The first inequality of the proposition then
follows as  $E_{T,k}^*(t,x)$ vanishes when $|x|\geq t+2^{k+1}.$

For the second inequality, again we rely on the following bound
\begin{align*}
\left|\nabla_v\left[\frac{\omega+\hat{v}}{1+\hat{v}\cdot\omega}\right]\right|\leq C\sqrt{1+|v|^2},
\end{align*}
which could be found, for example in \cite{Glassey87:VM:Hivelocity}.
Then using  Lemma \ref{lem:bd:support:fk:v}, we can show that
\begin{align*}
\numberthis\label{ESk}
|E_{S,k}^*(t,x)|&\leq C\int_{|y-x|\leq t}\int_{\R^3}\sqrt{1+|v|^2}(|E|+|B|)f_{[k]}|_{(t-|y-x|,y,v)}\frac{dvdy}{|y-x|}\\&\leq C\Lambda_{k,1}\|K\|_0\int_{|y-x|\leq t}\int_{\R^3}\frac{f_{[k]}(t-|y-x|,y,v)}{t-|y-x|+|y|+1}\frac{dvdy}{|y-x|}\\
&\leq \frac{C\Lambda_{k,1}\|K\|_0}{|t-|x||+1}\int_{|y-x|\leq t}\int_{\R^3}{f_{[k]}(t-|y-x|,y,v)}\frac{dvdy}{|y-x|} .
\end{align*}
Here the last step follows from the fact
\begin{align*}
t-|y-x|+|y|+1\geq |t-|x||+1,\quad \forall |y-x|\leq t.
\end{align*}
Then by Lemma \ref{lem:decay:fk} with $p=1$, we have
\begin{align*}
|E_{S,k}^*(t,x)|&\leq \frac{C2^{4k}\Lambda_{k,1}^4\Lambda_{k,2}\Lambda_{k,3}\|K\|_0\|f_{0,k}\|_0}{(|t-|x||+1)(2^k+t)}.
\end{align*}
The second inequality then follows by noting that $E_{S,k}^*(t,x)$ is supported on $\{|x|\leq t+2^{k+1}\}.$
\end{proof}

As a consequence of the above rough decay estimates, we now are able to control the $\|\cdot\|_{0}$ norm of the pure perturbation part of the Maxwell field.
\begin{Prop}
\label{prop:PureMaxwell:0norm}
If $ \Lambda> 2$, $K\in \mathcal{K}_{\Lambda}$,  then
\begin{align*}
\|(E_{S}^*,0)\|_0+\|(E_{T}^*,0)\|_0\leq C\Lambda^9(\ln\Lambda)^{11}\varepsilon_0.
\end{align*}
\end{Prop}
\begin{proof}
For positive integer $k$, recall the constants $\Lambda_{k,i}$ given in Proposition \ref{prop:bd:averaged:f}.  In particular
\begin{align*}
\Lambda_{k,i}&=2^{k}+ (\|K\|_0\ln(2+ \|K\|_0))^{i}\leq 2^{k}+ (\Lambda\ln(2+ \Lambda))^{i}\leq C(\Lambda\ln\Lambda)^{i}2^k,\quad i=1, 2,\\
 \Lambda_{k,3}&=\ln(1+\|K\|_{1+\alpha}+\|K\|_{0})+k+1\leq \ln(1+\Lambda^2+\Lambda)+k+1\leq Ck\ln\Lambda.
\end{align*}
Now by the previous Proposition \ref{prop:decay:Ek:pt} and the fact that
$$\|f_{0,k}\|_0\leq 2^{(2-k)q}\varepsilon_0,$$
we show that
\begin{align*}
|E_{T,k}^*(t,x)|&\leq C2^{2k}\Lambda_{k,1}^5\Lambda_{k,2}^2\Lambda_{k,3}^2(|x|+t+1)^{-2}\|f_{0,k}\|_0\\&\leq C2^{2k}(2^k\Lambda\ln\Lambda)^5((\Lambda\ln\Lambda)^22^k)^2(k\ln\Lambda)^2(|x|+t+1)^{-2}2^{-kq}\varepsilon_0\\&= Ck^22^{(9-q)k}\Lambda^9(\ln\Lambda)^{11}(|x|+t+1)^{-2}\varepsilon_0,
\\
|E_{S,k}^*(t,x)|&\leq \frac{C2^{4k}\Lambda_{k,1}^4\Lambda_{k,2}\Lambda_{k,3}\|K\|_0\|f_{0,k}\|_0}{(|t-|x||+1)(|x|+t+1)}\\&\leq \frac{C2^{4k}(2^k\Lambda\ln\Lambda)^4(\Lambda\ln\Lambda)^22^k(k\ln\Lambda)\Lambda2^{-kq}\varepsilon_0}{(|t-|x||+1)(|x|+t+1)}\\&= \frac{Ck2^{(9-q)k}\Lambda^7(\ln\Lambda)^{7}\varepsilon_0}{(|t-|x||+1)(|x|+t+1)}.
\end{align*}
Here the constant $C$ relies only on $q$.  By the decomposition \eqref{ETSk} and the assumption that $q>9 $, we conclude that
\begin{align*}
|E_{T}^*(t,x)|&\leq\sum_{k=1}^{+\infty}|E_{T,k}^*(t,x)| \\
&\leq\sum_{k=1}^{+\infty}Ck^22^{(9-q)k}\Lambda^9(\ln\Lambda)^{11}(|x|+t+1)^{-2}\varepsilon_0\\&\leq C\Lambda^9(\ln\Lambda)^{11}(|x|+t+1)^{-2}\varepsilon_0\\
&\leq C\Lambda^9(\ln\Lambda)^{11}(|t-|x||+1)^{-1}(|x|+t+1)^{-1}\varepsilon_0,
\\|E_{S}^*(t,x)|&\leq\sum_{k=1}^{+\infty}|E_{S,k}^*(t,x)| \\
&\leq\sum_{k=1}^{+\infty} \frac{Ck2^{(9-q)k}\Lambda^7(\ln\Lambda)^{7}\varepsilon_0}{(|t-|x||+1)(|x|+t+1)}\\
&\leq\frac{C\Lambda^7(\ln\Lambda)^{7}\varepsilon_0}{(|t-|x||+1)(|x|+t+1)}.
\end{align*}
These estimates lead to
\begin{align*}\|(E_{S}^*,0)\|_0+\|(E_{T}^*,0)\|_0\leq C\Lambda^7(\ln\Lambda)^{7}\varepsilon_0+C\Lambda^9(\ln\Lambda)^{11}\varepsilon_0\leq C\Lambda^9(\ln\Lambda)^{11}\varepsilon_0.
\end{align*}
This completes the proof for the proposition.
\end{proof}

To close the argument, we also need to control the mixed Lipschitz norm $\|\cdot\|_{1+\a}$ of the Maxwell field. This requires refined decay estimates.
\begin{Prop}
\label{prop:PureMaxwell:anorm:ESK}
For positive integer $k$, recall constants $\Lambda_{k,i}$ defined in Proposition \ref{prop:bd:averaged:f}. Then for $t\geq 0,\ x\in\R^3,\ |x|\leq t$ we have the improved decay estimate
\begin{align*}
|E_{S,k}^*(t,x)|&\leq \frac{C\Lambda_{k,1}^6\Lambda_{k,2}^3\Lambda_{k,3}^3\|K\|_0\|f_{0,k}\|_0\ln(t-|x|+2)}{(t-|x|+1)^2(|x|+t+1)}.
\end{align*}
\end{Prop}
\begin{proof}
 In view of the proof for Proposition \ref{prop:decay:Ek:pt} (equation \eqref{ESk}) and the definition of $\|K\|_0 $,  using the decay estimate for the charge density in Proposition \ref{prop:bd:averaged:f}, we can show that
\begin{align*}
 |E_{S,k}^*(t,x)|
& \leq C\Lambda_{k,1}\|K\|_0\int_{|y-x|\leq t}\int_{\R^3}\frac{f_{[k]}(t-|y-x|,y,v)}{(t-|y-x|+|y|+1)(|t-|y-x|-|y||+1)}\frac{dvdy}{|y-x|}\\
&\leq C\int_{|y-x|\leq t}\frac{\Lambda_{k,1}^6\Lambda_{k,2}^3\Lambda_{k,3}^3\|K\|_0\|f_{0,k}\|_0}{(t-|y-x|+|y|+1)^4(|t-|y-x|-|y||+1)}\frac{dy}{|y-x|}.
\end{align*}
Fix $r=|x|\leq t$ and define parameters $\la=|y|$, $\tau=t-|x-y|$. Changing variables for the above integration on the ball $|y-x|\leq t$ to the new parameters $(\la, \tau, \zeta)$, then integrate on the third variable. The above triple integral can be reduced to the following double integral
\begin{align}
\label{ESk1}
|E_{S,k}^*(t,x)|&\leq \frac{C\Lambda_{k,1}^6\Lambda_{k,2}^3\Lambda_{k,3}^3\|K\|_0\|f_{0,k}\|_0}{r}\int_0^t\int_a^b\frac{\lambda d\lambda d\tau}{(\tau+\lambda+1)^4(|\tau-\lambda|+1)}
\end{align}
with $a=|r-t+\tau|,\ b=r+t-\tau.$
For details about this procedure, we refer to Lemma 7 in \cite{Glassey87:VM:absence}.
  Let $u=\tau-\lambda,\ v=\tau+\lambda.$ Then the integral domain is
\begin{align*}
&\{0< \tau<t,\ a=|r-t+\tau|<\lambda<b=r+t-\tau\}\\=&\{0< \tau<t,\ -\lambda<r-t+\tau<\lambda<r+t-\tau\}\\=&\{0< \tau<t,\ t-r<\lambda+\tau,\ \tau-\lambda<t-r,\ \lambda+\tau<r+t\}\\=&\{0< (u+v)/2<t,\ t-r<v<r+t,\ u<t-r\}\\=&\{t-r<v<r+t,\ -v<u<t-r\}.
\end{align*}
Since $\lambda\leq\tau+\lambda+1 $, change variables $u=\tau-\lambda,\ v=\tau+\lambda$. We can show that
\begin{align*}
\int_0^t\int_a^b\frac{\lambda d\lambda d\tau}{(\tau+\lambda+1)^4(|\tau-\lambda|+1)}
&\leq \int_0^t\int_a^b\frac{ d\lambda d\tau}{(\tau+\lambda+1)^3(|\tau-\lambda|+1)}\\
&=\int_{t-r}^{r+t}\int_{-v}^{t-r}\frac{ du dv}{2(v+1)^3(|u|+1)}\\
&=\int_{t-r}^{r+t}\frac{ (\ln(v+1)+\ln(t-r+1)) dv}{2(v+1)^3}\\
&\leq C\frac{\ln(t-r+2)}{t-r+1}\int_{t-r}^{r+t}\frac{  dv}{2(v+1)^2}\\
&=\frac{C r\ln(t-r+2)}{(t-r+1)^2(r+t+1)}.
\end{align*}
This leads to
\begin{align*}
|E_{S,k}^*(t,x)|&\leq \frac{C\Lambda_{k,1}^6\Lambda_{k,2}^3\Lambda_{k,3}^3\|K\|_0\|f_{0,k}\|_0\ln(t-r+2)}{(t-r+1)^2(r+t+1)}.
\end{align*}
And the Proposition holds.
\end{proof}
Next we improve the bound for $E_{T, k}^{*}$.
\begin{Prop}
\label{prop:PureMaxwell:ETk:impr}
For positive integer $k$, let $\Lambda_{k,i}$ be constants defined in Proposition \ref{prop:bd:averaged:f}.  Then for $t\geq 2,\ x,y\in\R^3$ such that $|x|\leq|y|\leq t$, we have
\begin{align*}
\frac{|E_{T,k}^*(t,x)-E_{T,k}^*(t,y)|}{|x-y|+1}\leq & C(\|K\|_0+1)2^{2k}\Lambda_{k,1}^7\Lambda_{k,2}^2\Lambda_{k,3}^2\|f_{0,k}\|_0t^{-2}(t-|y|+1)^{-1}
\\&+ C\|f_{0,k}\|_0\Lambda_{k,1}^8\Lambda_{k,2}^3\Lambda_{k,3}^3t^{-3}\ln(t).
\end{align*}
\end{Prop}
\begin{proof}
For any interval $I\subset[0,t]$, let
\begin{align}
\label{ETkI}
E_{T,k,I}^*(t,x)=&-\int_{|y-x|\in I}\int_{\R^3}\frac{(\omega+\hat{v})(1-|\hat{v}|^2)}{(1+\hat{v}\cdot\omega)^2}f_{[k]}(t-|y-x|,y,v)\frac{dvdy}{|y-x|^2}.
\end{align}
In particular
 $E_{T,k}^*=E_{T,k,[0,1]}^*+E_{T,k,[1,t]}^*.$
In view of the inequality \eqref{om1} and the bound for the velocity in Lemma \ref{lem:bd:support:fk:v}, the decay estimate for the charge density in Proposition \ref{prop:bd:averaged:f} then leads to
\begin{equation}
\label{ETk01}
\begin{split}
|E_{T,k,[0,1]}^*(t,x)|&\leq C\int_{|y-x|\leq 1}\int_{\R^3}\sqrt{1+|v|^2}f_{[k]}(t-|y-x|,y,v)\frac{dvdy}{|y-x|^2}\\
&\leq C\Lambda_{k,1}\int_{|y-x|\leq 1}\frac{\|f_{0,k}\|_0\Lambda_{k,1}^5\Lambda_{k,2}^3\Lambda_{k,3}^3}{(t-|y-x|+1)^3}\frac{dy}{|y-x|^2}\\
&\leq C\|f_{0,k}\|_0\Lambda_{k,1}^6\Lambda_{k,2}^3\Lambda_{k,3}^3t^{-3} .
\end{split}
\end{equation}
For the main part $E_{T,k,[1,t]}^*$, inspired by the work \cite{Glassey:VM:singularity:86}, we make use of the equation for the Vlasov field. For interval $I=[a, b]$, we differentiate the representation formula \eqref{ETkI} for $E_{T,k,I}^*$ with respect to the $x$ variable. For $i,l\in\{1,2,3\}$, we have
\begin{align*}
\partial_{x_l} E_{T,k,I}^{*i}(t,x)=&-\int_{|\tilde{y}|\in I}\int_{\R^3}\frac{(\omega_i+\hat{v}_i)(1-|\hat{v}|^2)}{(1+\hat{v}\cdot\omega)^2}\partial_{x_l} f_{[k]}(t-|\tilde{y}|, x+\tilde{y},v)\frac{dvd\tilde{y}}{|\tilde{y}|^2}.
\end{align*}
The partial derivative $\partial_{x_l} $ can be expressed as linear combination of the characteristic operator $T_j=-\om_j \pa_t+\pa_{x_j}$ and the Vlasov operator $S=\partial_t+\hat{v}\cdot\nabla_y$. Using the Vlasov equation
\[
Sf=-(E+\hat{v}\times B)\cdot\nabla_vf=-\nabla_v\cdot[(E+\hat{v}\times B)f]
\]
and integration by parts in variables $y$ (or $\tilde{y}=y-x$) and $v$, we derive the following decomposition
\begin{align}
\label{ETkIl}
\partial_{x_l} E_{T,k,I}^{*i}=A_{w,k,I}+A_{TT,k,I}+A_{TS,k,I},
\end{align}in which
\begin{equation}
\label{ATSkI}
\begin{split}
A_{TT,k,I}=&\int_{|y-x|\in I}\int_{\R^3}a_{il}(\omega,\hat{v})f_{[k]}{(t-|y-x|,y,v)}\frac{dvdy}{|y-x|^3},\\
A_{w,k,I}=&s^{-2}\int_{|y-x|=s}\int_{\R^3}d(\omega,\hat{v})f_{[k]}{(t-s,y,v)}dvdS_y\Big|_{s=a}^{s=b},\\
 A_{TS,k,I}=&\int_{|y-x|\leq t}\int_{\R^3}d(\omega,\hat{v})Sf_{[k]}{(t-|y-x|,y,v)}\frac{dvdy}{|y-x|^2}\\
 =&\int_{|y-x|\leq t}\int_{\R^3}\nabla_v d(\omega,\hat{v})\cdot(E+\hat{v}\times B)f_{[k]}|_{(t-|y-x|,y,v)}\frac{dvdy}{|y-x|^2}.
\end{split}
\end{equation}
Here the kernels $a_{il}$ and $d$ are given by
\begin{align*}
a_{il}&=\frac{-3(\om_i+\hat{v}_i)(\om_l(1-|\hat{v}|^2)+\hat{v}_l(1+\om\cdot\hat{v}))+(1+\om\cdot \hat{v})^2 \delta_{il}}{(1+|v|^2)(1+\om\cdot \hat{v})^{4}},\\
d(\omega,\hat{v})&=-\frac{\omega_l(\omega_{i}+\hat{v}_{i})(1-|\hat{v}|^2)}{(1+\hat{v}\cdot\omega)^3}.
\end{align*}
The kernel $a_{il}$ is the same as that given in equation (33) in \cite{Glassey:VM:singularity:86}. For detailed computations regarding the above decomposition, we refer to Theorem 4 of the mentioned work of Glassey-Strauss. By direct computations (in view of inequality \eqref{om1}), the kernels verify the following upper bound
\begin{align}
\label{b1d}
|d(\omega,\hat{v})|+|a_{il}(\omega, \hat{v})|\leq C(1+|v|^2)^{3/2}.
\end{align}
Using the bound for the velocity of Lemma \ref{lem:bd:support:fk:v}, we therefore can bound that
 \begin{align*}
 |A_{TT,k,I}|\leq &C\int_{|y-x|\in I}\int_{\R^3}f_{[k]}{(t-|y-x|,y,v)}(1+|v|^2)^{3/2}\frac{dvdy}{|y-x|^3}\\
 \leq & C\Lambda_{k,1}^3\int_{|y-x|\in I}\int_{\R^3}f_{[k]}{(t-|y-x|,y,v)}\frac{dvdy}{|y-x|^3}.
\end{align*}
We split the space region into two parts. For the case when $|y-x|\geq \frac{t}{2}$, we rely on
Lemma \ref{lem:decay:fk} with $p=0$:
\begin{align*}
|A_{TT,k,[t/2,t]}|\leq &Ct^{-3}\Lambda_{k,1}^3\int_{|y-x|\leq t}\int_{\R^3}f_{[k]}{(t-|y-x|,y,v)}{dvdy}
\\ \leq& Ct^{-3}\Lambda_{k,1}^52^{6k}\|f_{0,k}\|_0.
\end{align*}
On the part when $1\leq |y-x|\leq \frac{t}{2}$, we instead use the charge density decay estimate of Proposition \ref{prop:bd:averaged:f} to show that
\begin{align*}
|A_{TT,k,[1,t/2]}|\leq &C\Lambda_{k,1}^3\int_{1\leq|y-x|\leq t/2}\|f_{0,k}\|_0(2^k+t/2)^{-3}\Lambda_{k,1}^5\Lambda_{k,2}^3\Lambda_{k,3}^3\frac{dy}{|y-x|^3}\\ \leq& C\|f_{0,k}\|_0(2^k+t)^{-3}\Lambda_{k,1}^8\Lambda_{k,2}^3\Lambda_{k,3}^3\ln(t).
\end{align*}
Since $\Lambda_{k,1}\geq 2^k$, $\Lambda_{k,2}\geq 2^k$, $\Lambda_{k,3}\geq 1$, we therefore conclude that
\begin{align}
\label{ATTk1t}
|A_{TT,k,[1,t]}|\leq &|A_{TT,k,[1,t/2]}|+|A_{TT,k,[t/2,t]}\leq C\|f_{0,k}\|_0t^{-3}\Lambda_{k,1}^8\Lambda_{k,2}^3\Lambda_{k,3}^3\ln(t).
\end{align}
With the upper bound on the kernel $d$ in \eqref{b1d}, we next control $A_{w, k, I}$, which consists of two integrals at time $s=t$ and $s=1$. For the integral at time $s=t$, in view of the support of the density distribution $f_{[k]}$, we can show that
\begin{align*}
&\left|s^{-2}\int_{|y-x|=s}\int_{\R^3}d(\omega,\hat{v})f_{[k]}{(t-s,y,v)}dvdS_y\Big|_{s=t}\right|\\ 
 &\leq  Ct^{-2}\int_{|y-x|=t}\int_{|v|\leq 2^{k+1}}(1+|v|^2)^{3/2}f_{0,k}{(y,v)}dvdS_y \\
 &\leq  Ct^{-2}2^{3k}\int_{|y-x|=t,|y|\leq 2^{k+1},|v|\leq 2^{k+1}}\|f_{0,k}\|_{0}dvdS_y\\
 &\leq \left\{\begin{array}{l@{\ \text{if}\ }l}
Ct^{-2}2^{3k}2^{2k}2^{3k}\|f_{0,k}\|_{0}& 0\leq t-|x|\leq 2^{k+1}\\
0&t-|x|\geq 2^{k+1}
\end{array}\right.\\
&\leq Ct^{-2}2^{9k}(t-|x|+1)^{-1}\|f_{0,k}\|_{0}.
\end{align*}
Here note that we have assumed $|x|\leq t$ in the Proposition.
For the integral at time $s=1$, using the decay of charge density of Proposition  \ref{prop:bd:averaged:f}, we have
 \begin{align*}
 &\left|s^{-2}\int_{|y-x|=s}\int_{\R^3}d(\omega,\hat{v})f_{[k]}{(t-s,y,v)}dvdS_y\Big|_{s=1}\right|\\
&\leq  C\int_{|y-x|=1}\int_{\R^3}(1+|v|^2)^{3/2}f_{[k]}{(t-1,y,v)}dvdS_y\\
 &\leq  C\Lambda_{k,1}^3\int_{|y-x|=1}\int_{\R^3}f_{[k]}{(t-1,y,v)}dvdS_y\\
 &\leq  C\|f_{0,k}\|_0t^{-3}\Lambda_{k,1}^8\Lambda_{k,2}^3\Lambda_{k,3}^3.
\end{align*}
In particular the above computations lead to
\begin{align*}
|A_{w,k,[1,t]}(t,x)|=&\left|s^{-2}\int_{|y-x|=s}\int_{\R^3}d(\omega,\hat{v})f_{[k]}{(t-s,y,v)}dvdS_y\Big|_{s=1}^{s=t}\right|\\
\numberthis\label{Awk1t}\leq& Ct^{-2}2^{9k}(|t-|x||+1)^{-1}\|f_{0,k}\|_{0}+C\|f_{0,k}\|_0t^{-3}\Lambda_{k,1}^8\Lambda_{k,2}^3\Lambda_{k,3}^3.
\end{align*}
Finally for $A_{TS, k, I}$, we need to bound the $v$-derivative of the kernel $d(\om, \hat{v})$. Using the inequality \eqref{om1} again, we can compute that
\begin{align*}
&\frac{\partial}{\partial v_j}\frac{\omega_l(\omega_i+\hat{v}_i)(1-|\hat{v}|^2)}{(1+\hat{v}\cdot\omega)^3}=\frac{\partial}{\partial v_j}\frac{\omega_l(\sqrt{1+|v|^2}\omega_i+{v}_i)}{(\sqrt{1+|v|^2}+{v}\cdot\omega)^3}\\
&=\frac{\omega_l(\hat{v}_j\omega_i+\delta_{ij})}{(\sqrt{1+|v|^2}+{v}\cdot\omega)^3}-3\frac{\omega_l(\sqrt{1+|v|^2}\omega_i+{v}_i)}{(\sqrt{1+|v|^2}+{v}\cdot\omega)^4}
(\hat{v}_j+\omega_j)\\
&=
\frac{\omega_l(\hat{v}_j\omega_i+\delta_{ij})}{(1+|v|^2)^{3/2}(1+\hat{v}\cdot\omega)^3}-
3\frac{\omega_l(\omega_i+\hat{v}_i)(\hat{v}_j+\omega_j)}{(1+|v|^2)^{3/2}(1+\hat{v}\cdot\omega)^4}.
\end{align*}
Note that
\begin{align*}
|\omega+\hat{v}|^2\leq 2(1+\hat{v}\cdot\omega),\quad 1+\hat{v}\cdot \om \geq \frac{1}{2}(1+|v|^2)^{-1}.
\end{align*}
The above computations imply that
\begin{align*}
&\left| \nabla_v d(\omega,\hat{v})\right|\leq
\frac{8}{(1+|v|^2)^{3/2}(1+\hat{v}\cdot\omega)^3}\leq 64(1+|v|^2)^{3/2}.
\end{align*}
Since
\[
t-|y-x|+|y|\geq t-|x|\geq 0,
\]
by the definition  \eqref{ATSkI},  Lemma \ref{lem:bd:support:fk:v} and Lemma \ref{lem:decay:fk}, we show that
\begin{align*}
|A_{TS,k,I}(t,x)|&\leq C\int_{|y-x|\leq t}\int_{\R^3}({1+|v|^2})^{3/2}(|E|+|B|)f_{[k]}|_{(t-|y-x|,y,v)}\frac{dvdy}{|y-x|^2}\\&\leq C\Lambda_{k,1}^3\|K\|_0\int_{|y-x|\leq t}\int_{\R^3}\frac{f_{[k]}(t-|y-x|,y,v)}{t-|y-x|+|y|+1}\frac{dvdy}{|y-x|^2}
\\&\leq \frac{C\Lambda_{k,1}^3\|K\|_0}{t-|x|+1}\int_{|y-x|\leq t}\int_{\R^3}{f_{[k]}(t-|y-x|,y,v)}\frac{dvdy}{|y-x|^2}
\\
&\leq \frac{C\Lambda_{k,1}^3\|K\|_0}{t-|x|+1}2^{2k}\Lambda_{k,1}^4\Lambda_{k,2}^2\Lambda_{k,3}^2t^{-2}\|f_{0,k}\|_0.
\end{align*}
In view of the decomposition \eqref{ETkIl}, combining the above estimates \eqref{ATTk1t}, \eqref{Awk1t},
we conclude that
\begin{align*}
&|\nabla_xE_{T,k,[1,t]}^*(t,x)|\leq CA_1t^{-3}\ln(t)+CA_2t^{-2}(t-|x|+1)^{-1},\ \forall\ |x|\leq t,\ t\geq 2
\end{align*}
with constants
\begin{align*}
&A_1=\|f_{0,k}\|_0\Lambda_{k,1}^8\Lambda_{k,2}^3\Lambda_{k,3}^3,\quad
A_2=(\|K\|_0+1)2^{2k}\Lambda_{k,1}^7\Lambda_{k,2}^2\Lambda_{k,3}^2\|f_{0,k}\|_0.
\end{align*}
Here we keep in mind that  $\Lambda_{k,1}\geq 2^k$, $\Lambda_{k,2}\geq 1$ and $\Lambda_{k,3}\geq 1$.
Then applying the mean value theorem to $E_{T,k,[1,t]}$, it follows that
\begin{align*}
&|E_{T,k,[1,t]}^*(t,x)-E_{T,k,[1,t]}^*(t,y)|\\
&\leq |x-y|\sup_{s\in[0,1]}|\nabla_xE_{T,k,[1,t]}^*(t,sx+(1-s)y)|\\
&\leq C|x-y|(A_1t^{-3}\ln(t)+A_2t^{-2}(t-|y|+1)^{-1}),
\end{align*}
which together with \eqref{ETk01} and $E_{T,k}^*=E_{T,k,[0,1]}^*+E_{T,k,[1,t]}^* $ gives
\begin{align*}
&|E_{T,k}^*(t,x)-E_{T,k}^*(t,y)|/(|x-y|+1)\\
&\leq \frac{|E_{T,k,[0,1]}^*(t,x)|+|E_{T,k,[0,1]}^*(t,y)|+|E_{T,k,[1,t]}^*(t,x)-E_{T,k,[1,t]}^*(t,y)|}{|x-y|+1}\\
&\leq \frac{C\Lambda_{k,1}^6\Lambda_{k,2}^3\Lambda_{k,3}^3t^{-3}\|f_{0,k}\|_0+C|x-y|(A_1t^{-3}\ln(t)+A_2t^{-2}(t-|y|+1)^{-1})}{|x-y|+1}\\
&\leq {C\Lambda_{k,1}^8\Lambda_{k,2}^3\Lambda_{k,3}^3\|f_{0,k}\|_0t^{-3}\ln(t)+C(A_1t^{-3}\ln(t)+A_2t^{-2}(t-|y|+1)^{-1})}.
\end{align*}
This completes the proof by recalling the above definition of the constants $A_1,\ A_2.$
\end{proof}
With all the above preparations, we are now ready to bound the mixed Lipschitz norm $\|\cdot\|_{1+\a}$ for the electric field.
\begin{Prop}
\label{prop:PureMaxwell:anorm:bound}
For any $K\in \mathcal{K}_{\Lambda}$ with $ \Lambda> 2$, we can bound that
\begin{align*}
\|(E_{S}^*,0)\|_{1+\alpha}+\|(E_{T}^*,0)\|_{1+\alpha}\leq C\Lambda^{10}(\ln\Lambda)^{11}\varepsilon_0.
\end{align*}
\end{Prop}
\begin{proof}
For positive integers $k$, recall the constants $\Lambda_{k,i}$ given in Proposition \ref{prop:bd:averaged:f}. From the proof of Proposition \ref{prop:PureMaxwell:0norm}, we know that
\begin{align}
\label{Lak}
\Lambda_{k,1}\leq C(\Lambda\ln\Lambda)2^k,\quad\Lambda_{k,2}\leq C(\Lambda\ln\Lambda)^{2}2^k,\quad\Lambda_{k,3}\leq Ck\ln\Lambda.
\end{align}
In particular $$\Lambda_{k,1}^6\Lambda_{k,2}^3\Lambda_{k,3}^3\leq Ck^32^{9k}\Lambda^{12}(\ln\Lambda)^{15}.$$
By the previous Proposition \ref{prop:PureMaxwell:anorm:ESK} and the fact that $\|f_{0,k}\|_0\leq 2^{(2-k)q}\varepsilon_0$, for $|x|\leq t$,  we have
\begin{align*}
|E_{S,k}^*(t,x)|&\leq \frac{C\Lambda_{k,1}^6\Lambda_{k,2}^3\Lambda_{k,3}^3\|K\|_0\|f_{0,k}\|_0\ln(t-|x|+2)}{(t-|x|+1)^2(|x|+t+1)}\\
&\leq \frac{Ck^32^{9k}\Lambda^{12}(\ln\Lambda)^{15}\Lambda2^{-kq}\varepsilon_0\ln(t-|x|+2)}{(t-|x|+1)^2(|x|+t+1)}\\
&= \frac{Ck^32^{(9-q)k}\Lambda^{13}(\ln\Lambda)^{15}\varepsilon_0\ln(t-|x|+2)}{(t-|x|+1)^2(|x|+t+1)}.
\end{align*}
In view of the decomposition \eqref{ETSk} and the assumption $q>9$, $\beta\in(0,1) $ we then can derive that
\begin{align*}
|E_{S}^*(t,x)|&\leq\sum_{k=1}^{+\infty}|E_{S,k}^*(t,x)| \\
&\leq\sum_{k=1}^{+\infty} \frac{Ck^32^{(9-q)k}\Lambda^{13}(\ln\Lambda)^{15}\varepsilon_0\ln(t-|x|+2)}{(t-|x|+1)^2(|x|+t+1)}
\\
&\leq\frac{C\Lambda^{13}(\ln\Lambda)^{15}\varepsilon_0\ln(t-|x|+2)}{(t-|x|+1)^2(|x|+t+1)}\\
&\leq\frac{C\Lambda^{13}(\ln\Lambda)^{15}\varepsilon_0}{(t-|x|+1)^{1+\beta}(|x|+t+1)}.
\end{align*}
From the proof of Proposition \ref{prop:PureMaxwell:0norm}, we also have that
\begin{align*}
|E_{S}^*(t,x)|&\leq\frac{C\Lambda^7(\ln\Lambda)^{7}\varepsilon_0}{(|t-|x||+1)(|x|+t+1)}.
\end{align*}
Thus for $|x|\leq t$ and $ 0<\alpha<\beta/6$, we demonstrate that
\begin{align*}
|E_{S}^*(t,x)|&\leq\min\left(\frac{C\Lambda^7(\ln\Lambda)^{7}\varepsilon_0}{(|t-|x||+1)(|x|+t+1)},
\frac{C\Lambda^{13}(\ln\Lambda)^{15}\varepsilon_0}{(t-|x|+1)^{1+\beta}(|x|+t+1)}\right)\\
&\leq\frac{C\Lambda^7(\ln\Lambda)^{7}\varepsilon_0}{(t-|x|+1)(|x|+t+1)}\min\left(1,
\frac{\Lambda^{6}(\ln\Lambda)^{8}}{(t-|x|+1)^{\beta}}\right)\\
&\leq\frac{C\Lambda^7(\ln\Lambda)^{7}\varepsilon_0}{(t-|x|+1)(t+1)}
\frac{\Lambda^{6\alpha/\beta}(\ln\Lambda)^{8\alpha/\beta}}{(t-|x|+1)^{\alpha}}
\\
&\leq\frac{C\Lambda^8(\ln\Lambda)^{9}\varepsilon_0}{(t-|x|+1)^{1+\alpha}(t+1)}.
\end{align*}
Hence for $|x|\leq|y|\leq t$, we can show that
\begin{align*}
\frac{|E_{S}^*(t,x)-E_{S}^*(t,y)|}{|x-y|+1} & \leq|E_{S}^*(t,x)|+|E_{S}^*(t,y)|\\
& \leq \frac{C\Lambda^8(\ln\Lambda)^{9}\varepsilon_0}{(t-|x|+1)^{1+\alpha}(t+1)}+\frac{C\Lambda^8(\ln\Lambda)^{9}\varepsilon_0}{(t-|y|+1)^{1+\alpha}(t+1)}\\
&\leq\frac{C\Lambda^8(\ln\Lambda)^{9}\varepsilon_0}{(t-|y|+1)^{1+\alpha}(t+1)},
\end{align*}
which implies $\|(E_{S}^*,0)\|_{1+\alpha}\leq C\Lambda^8(\ln\Lambda)^{9}\varepsilon_0$.

Next we consider $E_{T}^{*}$. Proposition  \ref{prop:PureMaxwell:0norm} shows that
\begin{align*}
|E_{T}^*(t,x)|&\leq C\Lambda^9(\ln\Lambda)^{11}(|x|+t+1)^{-2}\varepsilon_0\leq C\Lambda^{10}(\ln\Lambda)^{11}\varepsilon_0,\\|E_{T,k}^*(t,x)|&\leq Ck^22^{(9-q)k}\Lambda^9(\ln\Lambda)^{11}(|x|+t+1)^{-2}\varepsilon_0\\&\leq Ck^22^{(9-q)k}\Lambda^9(\ln\Lambda)^{11}(t+1)^{-2}\varepsilon_0.
\end{align*}
Thus if $|x|\leq |y|\leq t\leq 2 $, it trivially holds that
\begin{align*}
\numberthis\label{ETt<2}
&(t-|y|+1)^{1+\alpha}(t+1)
{|E_{T}^*(t,x)-E_{T}^*(t,y)|}/({|x-y|+1})
\leq C\Lambda^{10}(\ln\Lambda)^{11}\varepsilon_0.
\end{align*}
It hence remains to consider the case when $|x|\leq |y|\leq t$, with $t\geq 2 $.  On one hand, the above decay estimates for $E_{T, k}^{*}$ imply that
\begin{align*}
\numberthis\label{ETk1}
{|E_{T,k}^*(t,x)-E_{T,k}^*(t,y)|}/({|x-y|+1}) 
&\leq  Ck^22^{(9-q)k}\Lambda^9(\ln\Lambda)^{11}(t+1)^{-2}\varepsilon_0.
\end{align*}
On the other hand in view of  \eqref{Lak}, we in particular have that
\begin{align*}
&\Lambda_{k,1}^8\Lambda_{k,2}^3\Lambda_{k,3}^3\leq Ck^32^{11k}\Lambda^{14}(\ln\Lambda)^{17},\quad
\Lambda_{k,1}^7\Lambda_{k,2}^2\Lambda_{k,3}^2\leq Ck^22^{9k}\Lambda^{11}(\ln\Lambda)^{13}.
\end{align*}
Then from the previous Proposition \ref{prop:PureMaxwell:ETk:impr} and the bound $\|f_{0,k}\|_0\leq 2^{(2-k)q}\varepsilon_0$, we conclude that
\begin{align*}
&|E_{T,k}^*(t,x)-E_{T,k}^*(t,y)|/(|x-y|+1)\\
& \leq C\|f_{0,k}\|_0\Lambda_{k,1}^8\Lambda_{k,2}^3\Lambda_{k,3}^3t^{-3}\ln(t)
+C(\|K\|_0+1)2^{2k}\Lambda_{k,1}^7\Lambda_{k,2}^2\Lambda_{k,3}^2\|f_{0,k}\|_0t^{-2}(t-|y|+1)^{-1}\\
&\leq C2^{-kq}\varepsilon_0k^32^{11k}\Lambda^{14}(\ln\Lambda)^{17}t^{-3}\ln(t)
+C(\Lambda+1)2^{2k}k^22^{9k}\Lambda^{11}(\ln\Lambda)^{13}2^{-kq}\varepsilon_0t^{-2}(t-|y|+1)^{-1}\\
& \leq C\varepsilon_0k^32^{(11-q)k}\Lambda^{14}(\ln\Lambda)^{17}(t^{-3}\ln(t)+t^{-2}(t-|y|+1)^{-1})\\
&\leq C\varepsilon_0k^32^{(11-q)k}\Lambda^{14}(\ln\Lambda)^{17}t^{-2}(t-|y|+1)^{-\beta}
\end{align*}
as $\beta\in(0,1),\ t\geq 2,\ t\geq |y| $.
Optimize this inequality with \eqref{ETk1}. For $0<\alpha<\beta/6$, we derive that
\begin{align*}
&|E_{T,k}^*(t,x)-E_{T,k}^*(t,y)|/(|x-y|+1)\\
& \leq \min(Ck^22^{(9-q)k}\Lambda^9(\ln\Lambda)^{11}t^{-2}\varepsilon_0,C\varepsilon_0k^32^{(11-q)k}\Lambda^{14}(\ln\Lambda)^{17}t^{-2}(t-|y|+1)^{-\beta})\\
& \leq Ck^32^{(9-q)k}\Lambda^9(\ln\Lambda)^{11}t^{-2}\varepsilon_0\min(1,\quad 2^{2k}\Lambda^{5}(\ln\Lambda)^{6}(t-|y|+1)^{-\beta})\\
& \leq Ck^32^{(9-q)k}\Lambda^9(\ln\Lambda)^{11}t^{-2}\varepsilon_0\min(1,\quad 2^{2k}\Lambda^{6}(t-|y|+1)^{-\beta})\\
& \leq Ck^32^{(9-q)k}\Lambda^9(\ln\Lambda)^{11}t^{-2}\varepsilon_02^{2k\alpha/\beta}\Lambda^{6\alpha/\beta}(t-|y|+1)^{-\alpha}\\
& \leq Ck^32^{(9+2\alpha/\beta-q)k}\Lambda^{10}(\ln\Lambda)^{11}t^{-2}\varepsilon_0(t-|y|+1)^{-\alpha}.
\end{align*}
Recall that $\a$ is a small constant such that $0<\alpha<(q-9)\beta/2$. In particular we have
$$9+2\alpha/\beta-q<0. $$
In view of the decomposition \eqref{ETSk}, for $t\geq 2$ and $|x|\leq |y|\leq t$, we have
\begin{align*}
\frac{|E_{T}^*(t,x)-E_{T}^*(t,y)|}{|x-y|+1} & \leq\sum_{k=1}^{+\infty}\frac{|E_{T,k}^*(t,x)-E_{T,k}^*(t,y)|}{|x-y|+1}\\
& \leq \sum_{k=1}^{+\infty} Ck^32^{(9+2\alpha/\beta-q)k}\Lambda^{10}(\ln\Lambda)^{11}t^{-2}\varepsilon_0(t-|y|+1)^{-\alpha}\\
&  \leq
C\varepsilon_0\Lambda^{10}(\ln\Lambda)^{11}t^{-2}(t-|y|+1)^{-\alpha}\\
& \leq
C\varepsilon_0\Lambda^{10}(\ln\Lambda)^{11}(t+1)^{-1}(t-|y|+1)^{-1-\alpha},
\end{align*}
which together with \eqref{ETt<2} for the case when $t\leq 2$ implies that
 $$\|(E_{T}^*,0)\|_{1+\alpha}\leq C\Lambda^{10}(\ln\Lambda)^{11}\varepsilon_0.$$
 Combining this with the above bound for $E_{S}^*$, we have shown that
\begin{align*}
\|(E_{S}^*,0)\|_{1+\alpha}+\|(E_{T}^*,0)\|_{1+\alpha}\leq C\Lambda^8(\ln\Lambda)^{9}\varepsilon_0+C\Lambda^{10}(\ln\Lambda)^{11}\varepsilon_0\leq C\Lambda^{10}(\ln\Lambda)^{11}\varepsilon_0.
\end{align*}
This completes the proof for the Proposition.
\end{proof}

With all the above preparations, we are now able to show that for sufficiently small $\varepsilon_0$, the map $(E, B)\rightarrow (E^*, B^*)$ is closed.
\begin{Prop}
\label{prop:bd4:Kstar}
There exists a small constant $0<\varepsilon_* <1 $ such that if
\begin{align*}
\varepsilon_0\leq \varepsilon_* \Lambda^{-8}(\ln \Lambda)^{-11} ,\quad \Lambda\geq 2+2\|(\mathcal{E},\mathcal{B})\|,
\end{align*}
 then $K^*\in \mathcal{K}_{\Lambda}$ for all $K\in \mathcal{K}_{\Lambda}$.
\end{Prop}
 \begin{proof}
  In view of Proposition \ref{prop:bd:Ezstar}, Proposition \ref{prop:PureMaxwell:0norm} and Proposition \ref{prop:PureMaxwell:anorm:bound} and we have
\begin{align*}
\|(E^*-\mathcal{E},0)\|_{0}&\leq\|(E_z^*-\mathcal{E},0)\|_{0}+\|(E_{S}^*,0)\|_{0}+\|(E_{T}^*,0)\|_{0}\\
&\leq C\varepsilon_0+C\Lambda^{9}(\ln\Lambda)^{11}\varepsilon_0\\
& \leq C\Lambda^{9}(\ln\Lambda)^{11}\varepsilon_0,
\\
\|(E^*-\mathcal{E},0)\|_{1+\alpha}&\leq\|(E_z^*-\mathcal{E},0)\|_{1+\alpha}+\|(E_{S}^*,0)\|_{1+\alpha}+\|(E_{T}^*,0)\|_{1+\alpha}\\
&\leq C\varepsilon_0+C\Lambda^{10}(\ln\Lambda)^{11}\varepsilon_0\\
&\leq C\Lambda^{10}(\ln\Lambda)^{11}\varepsilon_0.
\end{align*}
Here recall that $E^*=E_z^*+E_T^*+E_S^*$.
Similarly for the magnetic field $B$, we also have
\begin{align*}
\|(0,B^*-\mathcal{B})\|_{0}&\leq C\Lambda^{9}(\ln\Lambda)^{11}\varepsilon_0,
\quad \|(0,B^*-\mathcal{B})\|_{1+\alpha}\leq C\Lambda^{10}(\ln\Lambda)^{11}\varepsilon_0.
\end{align*}
Thus there exists a constant $C_1>1$ (depending only on $q,\ \alpha,\ \beta$) such that
\begin{align*}
&\|(E^*-\mathcal{E},B^*-\mathcal{B})\|_{0}\leq C_1\Lambda^{9}(\ln\Lambda)^{11}\varepsilon_0,
\\& \|(E^*-\mathcal{E},B^*-\mathcal{B})\|_{1+\alpha}\leq C_1\Lambda^{10}(\ln\Lambda)^{11}\varepsilon_0.
\end{align*}
Now take $$ \varepsilon_*=(2C_1)^{-1}. $$
Then for $ \varepsilon_0\leq \varepsilon_* \Lambda^{-8}(\ln \Lambda)^{-11}  $, we conclude that
\begin{align*}
&\|(E^*-\mathcal{E},B^*-\mathcal{B})\|_{0}\leq \f12 \Lambda ,
\quad \|(E^*-\mathcal{E},B^*-\mathcal{B})\|_{1+\alpha}\leq \f12 \Lambda^{2}.
\end{align*}
Note that
\[
\Lambda\geq 2+2\|(\mathcal{E},\mathcal{B})\|=2+2(\|(\mathcal{E},\mathcal{B})\|_0+\|(\mathcal{E},\mathcal{B})\|_{1+\alpha}).
\]
We then can bound that
\begin{align*}
&\|(E^*,B^*)\|_{0}\leq \|(\mathcal{E},\mathcal{B})\|_0+\|(E^*-\mathcal{E},B^*-\mathcal{B})\|_{0}\leq \f12 \Lambda+\f12 \Lambda=\Lambda,
\\
& \|(E^*,B^*)\|_{1+\alpha}\leq \|(\mathcal{E},\mathcal{B})\|_{1+\alpha}+\|(E^*-\mathcal{E},B^*-\mathcal{B})\|_{1+\alpha}\leq \f12 \Lambda+\f12 \Lambda^{2}\leq \Lambda^2.
\end{align*}
By definition, this shows that $K^*=(E^*,B^*)\in \mathcal{K}_{\Lambda}.$ We thus have shown the Proposition.
\end{proof}

\section{Proof for the main theorem}

Recall the iteration sequences $f^{(n)},\ K^{(n)}$ in the introduction.
By Proposition \ref{prop:bd4:lMaxwell} there exists a constant $C_0 $ (without loss of generality, we may assume that $C_0>1$) such that $$\|(\mathcal{E},\mathcal{B})\|\leq C_0M.$$
 Let
\[
\Lambda=2 C_0M+2,\quad \varepsilon_0=\Lambda^{-8}(\ln \Lambda)^{-11}\varepsilon_* ,
\]
where $\vep_*$ is the small positive constant in Proposition \ref{prop:bd4:Kstar}.

Since $K^{(0)}=(0, 0)\in \mathcal{K}_{\Lambda}$, in view of Proposition \ref{prop:bd4:Kstar}, we conclude that $ K^{(n)}\in \mathcal{K}_{\Lambda}$ for all $n$. Then from Proposition \ref{prop:bd4f}, we derive that
$$\int_{\R^3}|{v}|f^{(n)}(t,x,v)dv\leq CC(\Lambda),\quad \forall n. $$
 This together with the main result of Glassey-Strauss in Theorem \ref{thm:strauss89}, we conclude that the sequence $(f^{(n)},\ K^{(n)})$ and their first derivatives converge pointwise to $(f,K)$ which solves the nonlinear system RVM. In particular
\begin{align*}
(\left|t-|x|\right|+1)(t+|x|+1)(|E(t,x)|+|B(t,x)|)\leq \|K\|_0\leq \Lambda =2C_0M+2<4C_0M
\end{align*}
for all $t\geq 0,\ x\in\R^3$ as $K\in \mathcal{K}_{\Lambda}$.

Moreover from Proposition \ref{prop:bd:averaged:f}, we derive that
\begin{align*}
\int_{\mathbb{R}^3}f_{[k]}(t, x, v)dv & \leq C\|f_{0, k}\|_{0}  (t+|x|+1)^{-3} \La_{k, 1}^5 \La_{k, 2}^3 \La_{k, 3}^3\\
&\leq C \varepsilon_0 2^{(2-k)q} (t+|x|+1)^{-3} (2^k \La \ln \La)^5 (2^k \La^2 (\ln\La)^2)^3 (1+k+\ln\La)^3\\
&\leq C \varepsilon_0 2^{(8-q)k } k^3 (t+|x|+1)^{-3}\La^{11} (\ln\La)^{14}.
\end{align*}
Since $q>9$ and using the decomposition of $f$ in section \ref{sec:localization}, we therefore derive that
\begin{align*}
\int_{\mathbb{R}^3} f(t, x, v)dv &= \sum\limits_{k=0}^{\infty} \int_{\mathbb{R}^3} f_{[k]}(t, x, v)dv\\
&\leq  C \varepsilon_0 (1+t+|x|)^{-3} \La^{11} (\ln\La)^{14}.
\end{align*}
We thus finished the proof for the main Theorem \ref{thm:main}.

\bibliography{shiwu}{}

\begin{thebibliography}{10}

\bibitem{Leo18:3D:VM}
L.~Bigorgne.
\newblock {Sharp asymptotic behavior of solutions of the 3d Vlasov-Maxwell
  system with small data}.
\newblock 2018.
\newblock ar{X}iv:1812.11897.

\bibitem{Lions89:globalweak:VM}
R.~DiPerna and P.~Lions.
\newblock Global weak solutions of {V}lasov-{M}axwell systems.
\newblock {\em Comm. Pure Appl. Math.}, 42(6):729--757, 1989.

\bibitem{Jacques17:Vect:Vlasov}
D.~Fajman, J.~Joudioux, and J.~Smulevici.
\newblock A vector field method for relativistic transport equations with
  applications.
\newblock {\em Anal. PDE}, 10(7):1539--1612, 2017.

\bibitem{Yang:mMKG:3D:largeMaxwell}
A.~Fang, Q.~Wang, and S.~Yang.
\newblock {Global solution for Massive Maxwell-Klein-Gordon equations with
  large Maxwell field}.
\newblock 2019.
\newblock ar{X}iv:1902.08927.

\bibitem{Glassey88:VM:neutral}
R.~Glassey and J.~Schaeffer.
\newblock Global existence for the relativistic {V}lasov-{M}axwell system with
  nearly neutral initial data.
\newblock {\em Comm. Math. Phys.}, 119(3):353--384, 1988.

\bibitem{Glassey90:1andhalfD:VM}
R.~Glassey and J.~Schaeffer.
\newblock On the ``one and one-half dimensional'' relativistic
  {V}lasov-{M}axwell system.
\newblock {\em Math. Methods Appl. Sci.}, 13(2):169--179, 1990.

\bibitem{Glassey:VM:2.5D}
R.~Glassey and J.~Schaeffer.
\newblock The ``two and one-half-dimensional'' relativistic {V}lasov {M}axwell
  system.
\newblock {\em Comm. Math. Phys.}, 185(2):257--284, 1997.

\bibitem{Glassey98:VM:2D:1}
R.~Glassey and J.~Schaeffer.
\newblock The relativistic {V}lasov-{M}axwell system in two space dimensions.
  {I}, {II}.
\newblock {\em Arch. Rational Mech. Anal.}, 141(4):331--354, 355--374, 1998.

\bibitem{Glassey:VM:singularity:86}
R.~Glassey and W.~Strauss.
\newblock Singularity formation in a collisionless plasma could occur only at
  high velocities.
\newblock {\em Arch. Rational Mech. Anal.}, 92(1):59--90, 1986.

\bibitem{Glassey87:VM:absence}
R.~Glassey and W.~Strauss.
\newblock Absence of shocks in an initially dilute collisionless plasma.
\newblock {\em Comm. Math. Phys.}, 113(2):191--208, 1987.

\bibitem{Glassey87:VM:Hivelocity}
R.~Glassey and W.~Strauss.
\newblock High velocity particles in a collisionless plasma.
\newblock {\em Math. Methods Appl. Sci.}, 9(1):46--52, 1987.

\bibitem{Glassey89:3DVM:largevelocity}
R.~Glassey and W.~Strauss.
\newblock Large velocities in the relativistic {V}lasov-{M}axwell equations.
\newblock {\em J. Fac. Sci. Univ. Tokyo Sect. IA Math.}, 36(3):615--627, 1989.

\bibitem{Horst90:VM:sphericalSy}
E.~Horst.
\newblock Symmetric plasmas and their decay.
\newblock {\em Comm. Math. Phys.}, 126(3):613--633, 1990.

\bibitem{Illner10:VM}
R.~Sospedra-Alfonsoand~R. Illner.
\newblock Classical solvability of the relativistic {V}lasov-{M}axwell system
  with bounded spatial density.
\newblock {\em Math. Methods Appl. Sci.}, 33(6):751--757, 2010.

\bibitem{Klainerman02:VM:Fourier}
S.~Klainerman and G.~Staffilani.
\newblock A new approach to study the {V}lasov-{M}axwell system.
\newblock {\em Commun. Pure Appl. Anal.}, 1(1):103--125, 2002.

\bibitem{Luk14:VM:criterion}
J.~Luk and R.~Strain.
\newblock A new continuation criterion for the relativistic {V}lasov-{M}axwell
  system.
\newblock {\em Comm. Math. Phys.}, 331(3):1005--1027, 2014.

\bibitem{Luk:VlasovMaxwell:Strichartz}
J.~Luk and R.~Strain.
\newblock Strichartz estimates and moment bounds for the relativistic
  {V}lasov-{M}axwell system.
\newblock {\em Arch. Ration. Mech. Anal.}, 219(1):445--552, 2016.

\bibitem{Pallar15:VM:criteria}
C.~Pallard.
\newblock A refined existence criterion for the relativistic {V}lasov-{M}axwell
  system.
\newblock {\em Commun. Math. Sci.}, 13(2):347--354, 2015.

\bibitem{Patel18:newcriteria:VM}
N.~Patel.
\newblock Three new results on continuation criteria for the 3{D} relativistic
  {V}lasov-{M}axwell system.
\newblock {\em J. Differential Equations}, 264(3):1841--1885, 2018.

\bibitem{Rein90:VM}
G.~Rein.
\newblock Generic global solutions of the relativistic {V}lasov-{M}axwell
  system of plasma physics.
\newblock {\em Comm. Math. Phys.}, 135(1):41--78, 1990.

\bibitem{Schaeffer:VM:notcompat:p}
J.~Schaeffer.
\newblock A small data theorem for collisionless plasma that includes high
  velocity particles.
\newblock {\em Indiana Univ. Math. J.}, 53(1):1--34, 2004.

\bibitem{Wang18:3D:VM}
X.~Wang.
\newblock { Propagation of regularity and long time behavior of the 3D massive
  relativistic transport equation II: Vlasov-Maxwell system}.
\newblock 2018.
\newblock ar{X}iv:1804.06566.

\bibitem{Wang20:3D:VM:large}
X.~Wang.
\newblock {Global solution of the 3D relativistic Vlasov-Maxwell system for the
  large radial data}.
\newblock 2020.
\newblock ar{X}iv:2003.14192.

\bibitem{wollman84:local:VM}
S.~Wollman.
\newblock An existence and uniqueness theorem for the {V}lasov-{M}axwell
  system.
\newblock {\em Comm. Pure Appl. Math.}, 37(4):457--462, 1984.

\bibitem{yangMKG}
S.~Yang.
\newblock Decay of solutions of {M}axwell-{K}lein-{G}ordon equations with
  arbitrary {M}axwell field.
\newblock {\em Anal. PDE}, 9(8):1829--1902, 2016.

\end{thebibliography}
\bibliographystyle{plain}
\end{document}